\documentclass[11pt]{amsart}
\usepackage{amssymb}
\usepackage{amsfonts}
\usepackage{amsmath}
\usepackage{amsthm}
\usepackage{dsfont}
\usepackage{mathtools}
\usepackage{geometry}
\usepackage[colorlinks=true,linkcolor=blue,urlcolor=blue]{hyperref}
\mathtoolsset{showonlyrefs}
\usepackage{natbib}
\usepackage{float}
\usepackage{comment}
\newtheorem{theorem}{Theorem}[section]
\newtheorem{remark}[theorem]{Remark}
\newtheorem{definition}[theorem]{Definition}

\newtheorem{lemma}[theorem]{Lemma}
\newtheorem{proposition}[theorem]{Proposition}
\newtheorem{corollary}[theorem]{Corollary}

\newcommand{\ud}{\mathrm d}
\newcommand{\cA}{\mathcal A}

\newcommand{\cC}{\mathcal C}
\newcommand{\cD}{\mathcal D}
\newcommand{\cF}{\mathcal F}
\newcommand{\cG}{\mathcal G}
\newcommand{\cJ}{\mathcal J}
\newcommand{\cL}{\mathcal L}
\newcommand{\cT}{\mathcal T}
\def\R{{\mathbb R}}

\def\e{{\mathrm e}}
\def\P{{\mathsf P}}

\def\E{{\mathsf E}}
\newcommand{\N}{\mathbb N}
 
\DeclareMathOperator*{\essinf}{ess\,inf}
\renewcommand{\mid}{|}
 
\numberwithin{equation}{section}
\usepackage{enumerate}

\title[A quickest detection problem with false negatives]{A quickest detection problem with false negatives}
\author[T.\ De Angelis, J.\ Garg, Q.\ Zhou]{Tiziano De Angelis \and Jhanvi Garg \and Quan Zhou}
\address{T.\ De Angelis: School of Management and Economics, Dept.\ ESOMAS, University of Torino, Corso Unione Sovietica, 218 Bis, 10134, Torino, Italy; Collegio Carlo Alberto, Piazza Arbarello 8, 10122, Torino, Italy.}
\email{\href{mailto:tiziano.deangelis@unito.it}{tiziano.deangelis@unito.it}}
\address{J.\ Garg: Department of Statistics, Texas A\&M University, 3143 TAMU, College Station, TX 77843.}
\email{\href{mailto:gargjhanvi@stat.tamu.edu}{gargjhanvi@stat.tamu.edu}}
\address{Q.\ Zhou: Department of Statistics, Texas A\&M University, 3143 TAMU, College Station, TX 77843.}
\email{\href{mailto:quan@stat.tamu.edu}{quan@stat.tamu.edu}}

\date{\today}
\numberwithin{equation}{section}

\usepackage{enumerate}
\begin{document}
\begin{abstract}
We formulate and solve  a variant  of the quickest detection problem  which features false negatives. A standard Brownian motion acquires a drift at an independent exponential random time which is not directly observable. Based on the observation in continuous time of the sample path of the process, an optimizer must detect the drift as quickly as possible after it has appeared. The optimizer can inspect the system multiple times upon payment of a fixed cost per inspection. If a test is performed on the system {\em before} the drift has appeared then, naturally, the test will return a negative outcome. However, if a test is performed {\em after} the drift has appeared, then the test may fail to detect it and return a false negative with probability $\epsilon\in(0,1)$. The optimization ends when the drift is eventually detected. The problem is formulated mathematically as an optimal multiple stopping problem, and it is shown to be equivalent to a recursive optimal stopping problem. Exploiting such connection and free boundary methods we find explicit formulae for the expected cost and the optimal strategy. We also show that when $\epsilon = 0$ our expected cost  is an affine transformation of the one in Shiryaev's classical optimal detection problem  with a rescaled model parameter.
\end{abstract}

\maketitle

\section{Introduction}
 \subsection{Classical quickest detection}

The classical quickest detection problem can be stated as follows. Let $B$ be a Brownian motion and $\theta$ be an independent exponentially distributed random variable. An optimizer observes the sample path of a stochastic process $X$ whose evolution is described by
\begin{equation}\label{eq:X-dynamics}
 \ud X_t=\left\{
\begin{array}{ll}
\ud B_t, & \text{ for } t\in[0,\theta),\\
\mu \ud t +\ud B_t, & \text{ for }  t\in[\theta,\infty).
\end{array}
\right.
\end{equation}
Here $\mu\in\R$ is known to the observer but $\theta$ and $B$ are not directly observable,  and the aim is to detect the appearance of the drift $\mu$ as quickly as possible.
This problem was originally motivated by the operation of radar systems, where $X$ models the real-time signal emitted by the radar (cf.\ Shiryaev \cite{shiryaev2010quickest} for a detailed historical account of the genesis of this problem and its various formulations).
More generally, {by Donsker's theorem}, model \eqref{eq:X-dynamics} can be viewed as a continuous-time approximation  {to the partial sums} of a discrete-time sequence of independent observations whose mean changes at an unknown time point.
{For that reason, \eqref{eq:X-dynamics} is often treated as a canonical model in the quickest detection literature.
Numerous applications have been found} in real-world change-point detection scenarios; see, e.g., the monographs by Poor and Hadjiliadis  \cite{poor2008quickest} and by Tartakovsky {\rm et al.} \cite{tartakovsky2014sequential}, {and the recent work by Johnson and Pedersen~\cite{johnson2025bayesian} on the application to epidemic models.}

Mathematically, the observer has access to the filtration $(\cF^X_t)_{t\ge 0}$ generated by the process $X$, i.e., $\cF^X_t=\sigma(X_s,\ s\le t)$.
In the classical setting, the observer chooses {\em one} $\cF^X_t$-stopping time $\tau$ in order to minimize a cost function
\begin{equation}\label{eq:classical-cost}
   \E[  \alpha 1_{\{ \tau<\theta \}} + \beta  (\tau-\theta)^+]
=  \alpha   \P(\tau<\theta) +  \beta\E[(\tau-\theta)^+],
\end{equation}
where $\P$ is a suitable probability measure on a suitable sample space $(\Omega, \cF)$,  $\alpha>0$ measures the cost of  a false alarm, and $\beta > 0$ measures the cost of delay in detection.
The literature on this class of problems is very vast and it dates back to Shiryaev's work in the 1960s (e.g., \cite{shiryaev1961a}, \cite{shiryaev1963}), where the problem above was formulated and solved for the first time. The finite-horizon version of the problem was studied by Gapeev and Peskir \cite{GP2006}, general one-dimensional diffusions were considered by Gapeev and Shiryaev \cite{gapeev2013bayesian} and, in particular, quickest detection for Bessel and Ornstein-Uhlenbeck processes was studied
by Peskir with Johnson~\cite{johnson2017quickest} and with Glover~\cite{PG2022}, respectively.
Recently, Peskir and Ernst \cite{ernst2022quickest} studied the detection of a drift in a coordinate of a multi-dimensional Brownian motion. Further references on this subject and the related ``disorder problem'' can be found in \cite[Chapter VI]{peskir2006optimal} and in Shiryaev's survey paper \cite{shiryaev2010quickest}.

\subsection{Motivating examples for our study}
To motivate our study and explain how it extends the classical quickest detection problem, it is convenient to think of the process $X$ as a signal emitted by a system and of $\theta$ as the time at which the operating mode of the system changes (e.g., because of a mechanical failure or because of an external disturbance).
The interpretation of the payoff in \eqref{eq:classical-cost} is as follows: at time $\tau$ the appearance of a drift is ``declared''; if $\tau<\theta$ the cost of a ``false alarm'' is $\alpha$, whereas if $\tau>\theta$ the cost of a ``delayed alarm'' is $\beta(\tau-\theta)$.
In order to calculate such costs, some sort of test must be performed on the system that reveals without error whether the drift has appeared or not.
Upon observing a negative outcome from the test,
in some applications one resumes monitoring the system until the next stopping decision is made
(e.g., in the radar operation problem discussed in~\cite{shiryaev2010quickest}).

In reality, a ``perfect'' inspection of the system is often impossible, particularly when the underlying disorder is challenging to identify, due to the possibility of {\em false negatives} resulting from the test.
In those cases, it makes sense to allow multiple (costly) inspections of the system.
Problems of this kind naturally arise in the healthcare context, where we may think of $X$ as an indicator of the general health condition of an individual and of $\theta$ as the time of the onset of some disease. The appearance of the drift $\mu$ in the dynamics of $X$ models the appearance of some symptoms, which may be difficult to detect at first or perhaps could be mistaken for some normal tiredness or stress. Indeed, it is well known
that the detection of complex diseases, like cancer in early stages, may require multiple tests and false negatives are not uncommon
(see, e.g., Bartlett {\rm et al.}~\citep{bartlett2021false} and Verbeek {\rm et al.}~\citep{verbeek2018acceptable}). Another example  is suggested by the extensive use of COVID-19 tests that we witnessed in recent years.
In that case, early symptoms could vary wildly across different individuals, and  it is not easy to tell them apart from those of a normal cold. Testing was key but the false negative rate of the PCR test can be as high as 10\%, according to Kanji {\rm et al.}~\citep{kanji2021false}.
{It is worth noticing that in our model the process $X$ is observed continuously and for free, whereas inspections are costly. This simplification is in line with the fact that many individuals have in their homes simple devices that allow to check several health indicators on a daily basis (e.g., body temperature, blood pressure, blood oxygen levels, etc.) Such devices can be purchased for a nearly negligible cost, and they last for decades --- thus we can think of them as ``for free''. On the contrary, COVID tests are single-use and generally not available for free. Likewise, testing for cancer or other complex diseases requires expensive equipment which is only available in hospitals or specialized clinics.}

{Quickest detection with imperfect inspections}
could also cover situations in which a device positioned in a remote location sends signals in continuous time to headquarters {(e.g., a measuring device installed on a satellite in space). The disorder time corresponds to the mechanical failure of one or more components in the device and it may be detected via a change in the form of the signal.}
{However, the latter may also indicate a temporary malfunctioning of the device (due to, e.g., a so-called ``single-event upset'', caused by ionizing particles striking the circuits of the device)}.
{An inspection corresponds to a remote reboot of the device and subsequent observation of the signal after the reboot. In the case of a temporary malfunctioning the reboot would suffice to restore the normal operating mode---the disorder has not occurred yet. If instead the device suffers from a faulty component, the reboot may still be successful and it may take some time before the headquarters detect new anomalies after the reboot---hence the inspection results in a false negative}.
{Thus, remote inspections are subject to error and may need to be repeated several times before the systematic disorder can be confirmed.}

It turns out that the mathematical literature has not yet considered quickest detection problems in such situations, and this work represents the first attempt in this direction.
For simplicity, we do not consider false positives in our  analysis, but in Section~\ref{sec:disc} we point to several generalizations of our model.
This simplification is not overly restrictive, because in many applications where a statistical hypothesis testing procedure is used to determine $1_{\{\tau < \theta\}}$, the false positive rate is controlled at a minimum level (for example, in the aforementioned work on COVID-19 testing \cite{kanji2021false}, the false positive rate is assumed to be zero).

\subsection{Our model and mathematical contribution}
To account for false negatives, we propose a new variant of the quickest detection problem.
We describe it informally here and we will rigorously formulate the problem in Section~\ref{sec:setting}.
We still assume the observed system's dynamics be modeled by the process $X$ given in~\eqref{eq:X-dynamics}.
Unlike the classical formulation, we assume that at the chosen stopping time $\tau$, the optimizer performs an imperfect inspection on the system. The outcome of such inspection is modeled as a binary random variable $Z \in \{0, 1\}$,  where $Z=1$ means that the drift is detected and $Z=0$ means otherwise. The conditional distribution of $Z$ is given according to  $\P(  Z = 0  \mid  \theta > \tau ) = 1$ and $\P( Z = 0  \mid  \theta \leq \tau  ) =  \epsilon$, where $\epsilon \in [0, 1)$ is referred to as the false negative rate and assumed known.
{Each inspection has a fixed cost,} and thus the optimizer wants to wisely choose when to perform the inspection.
{Without loss of generality, we set the cost of each inspection to $1$, as it will be clear from our objective function~\eqref{eq:our-cost} that any other value would merely rescale the parameter $\beta$.}
Upon  observing $Z = 0$ (which will be referred to as a {\em negative} outcome), the optimizer suitably adjusts her posterior belief of the presence of a drift, before resuming the observation of the system and choosing the next stopping time.
If $Z = 1$ (i.e., a {\em positive} outcome),  then the drift is successfully detected, and the optimization ends.
Letting $N$ denote the number of tests needed to detect the drift and $\tau_N$ denote the $N$th stopping time, the cost function to be minimized in our problem is
\begin{equation}\label{eq:our-cost}
\E[  N +   \beta(\tau_N - \theta)^+  ].
\end{equation}
Compared to the classical cost function \eqref{eq:classical-cost}, which penalizes false alarms directly through the false alarm probability,
our cost~\eqref{eq:our-cost} penalizes them using $\E[N]$.
{This enables us to explicitly model the effect of false negative outcomes from the inspections and to characterize how the optimal strategy depends on the false negative rate $\epsilon$  (via the positioning of the inspection threshold and the size of the jump in the posterior process).}
One key difference between our model formulation and the classical one is that in the classical quickest detection the optimization ends at the first time the alarm is declared, irrespective of whether or not the drift has actually appeared.
In our formulation, instead, the optimization ends when the drift is detected for the first time, as a result of one of many possible inspections.
Other versions of the problem,  for example those that
combine ``inspecting the system'' and ``declaring the alarm,'' are discussed in Section~\ref{sec:disc}.
Although in this work we primarily focus on how to minimize the cost function~\eqref{eq:our-cost}, in Section~\ref{sec:disc} we also discuss how to further generalize this problem to allow for false positives.

At the mathematical level the problem is formulated as an optimal multiple stopping problem (in the spirit of Kobylanski {\rm et al.}\ \cite{kobylanski2011optimal})
combined  with a progressive enlargement of filtrations  occurring at the stopping times when the inspections are performed.
Moreover, at the stopping times when the system's inspection returns a negative outcome, {\em jumps} are introduced to the posterior's dynamics.
To solve the problem, we show its equivalence to a recursive optimal stopping problem (somewhat in the spirit of Colaneri and De Angelis \cite{colaneri2021class}). The recursive problem has a natural link to a free boundary problem with recursive boundary conditions at the (unknown) optimal stopping boundary. We are able to solve explicitly the free boundary problem and fully identify both the optimal sequence of stopping times and the optimal cost in the quickest detection problem.  Our results are summarized in Theorem~\ref{thm:main}.

The structure of the mathematical problem is completely different from the one in the classical quickest detection framework. As a result, we obtain optimal strategies that cannot be derived simply by iterating the ones from Shiryaev's original work.
However, when $\epsilon = 0$, we are able to characterize a linear relationship between our expected cost  and that in Shiryaev's classical optimal detection problem, up to a scaling of the cost coefficient $\beta$, and show the equivalence between the two optimal policies.
We would like to emphasize that the explicit solution of a recursive optimal stopping problem is a significantly more challenging task than the solution of its non-recursive counterpart, and we believe it has never been accomplished in the literature. In particular,  Colaneri and De Angelis \cite{colaneri2021class} and Carmona and Touzi \cite{carmona2008swing} (see \cite[Section 4]{carmona2008swing}) deal, for different reasons, with optimal stopping problems with recursive structure but neither of them obtains an explicit formula for the optimal strategy or for the value function.

Recent work by Bayraktar {\rm et al.}\ \cite{BEG2022} considers a closely related quickest detection problem with costly observations. Unlike our setting, the optimizer in \cite{BEG2022} does not observe the sample path of $X$ in continuous time but only at discrete (stopping) times, upon paying a fixed cost per observation (in Bayraktar and Kravitz \cite{bayraktar2015quickest} the optimizer can make $n$ discrete observations at no cost).
Based on such discrete observations of the sample path, the optimizer must choose when to declare the appearance of the drift and, at that point, the optimization terminates and it is revealed whether $\theta$ has occurred or not. The key conceptual difference is that in our setup the ``observation'' of the system is for free but the ``inspection'' of the presence of a drift is costly, with limited accuracy and it can be repeated multiple times. In the setup of \cite{BEG2022} the observation of the system is costly and its close inspection can be performed only once but with perfect accuracy. The methodology in \cite{BEG2022} relies on an iterative construction of a fixed point for the cost function. A  characterization of an  optimal sequence of observation times is given in terms of a continuation region that cannot be found explicitly and whose geometry is investigated only numerically.
For completeness, we also mention that quickest detection problems with costly observations of the increments of $X$ were considered by
Antelman and Savage \cite{antelman1965surveillance}, Balmer \cite{B75} and Dalang and Shiryaev \cite{DS15}. The cost depends on the length of the observation, and the problems take the form of stochastic control with discretionary stopping; that leads to a framework completely different from ours.

{Finally, we observe that the solution to our problem  resembles the multi-stage sequential testing rules studied in Shiryaev~\cite{shiryaev1961a}, which are also surveyed in Shiryaev~\cite{shiryaev2010quickest}.
Sequential testing can be viewed as a special case of quickest detection where $\theta$ equals either $0$ or $\infty$. In the multi-stage setting, the same stopping rule  (designed for the single-stage problem) is applied repeatedly, and the stopping boundary can be chosen to optimize a long-term performance criterion, such as the expected detection delay subject to a fixed mean time to the first false alarm. The reset of the posterior process in this context is part of the problem formulation, hence exogenous. Our problem is fundamentally different:
the objective function \eqref{eq:our-cost} is not a concatenation of single-stage optimizations, because from outset our agent takes into account the occurrence of false negatives and she looks for a stopping rule that minimizes ``globally'' the cost criterion. That leads to  a recursive structure of the problem in which each negative inspection generates {\em endogenously} a downward jump in the posterior, and the optimization continues within the same problem.}

\subsection{Structure of the paper}

Our paper is organized as follows. In Section~\ref{sec:setting} we introduce the mathematical model in its optimal multiple stopping formulation, with a  progressively  enlarging filtration, and we describe the dynamics of the posterior process (which includes jumps). In Section~\ref{sec:mainresults} we state our main result (Theorem~\ref{thm:main}) containing the explicit form of the value function of our problem and the optimal multiple stopping rule.
In Section~\ref{sec:recursive}, we define a suitable recursive optimal stopping problem and we show the equivalence between the original optimal multiple stopping problem and the recursive one. In Section~\ref{sec:solution} we obtain the explicit solution of the recursive problem via the solution of a free boundary problem with recursive boundary conditions at the free boundary.
In Section~\ref{sec:solQDP} we give the formal proof of our main result. We discuss properties of the optimal strategy and illustrate its dependence on various model parameters also with the support of a detailed numerical study.  Section~\ref{sec:disc} contains a short discussion about possible generalizations of our model.  The paper is completed by a technical Appendix.

\section{Setting}\label{sec:setting}

We consider a probability space $(\Omega, \cF,\P_\pi)$ sufficiently rich to host a Brownian motion $(B_t)_{t\ge 0}$, a random variable $\theta\in[0,\infty)$ and an infinite sequence of i.i.d.\ random variables $(U_k)_{k=1}^\infty$ with $U_1\sim U(0,1)$, the uniform distribution on $[0,1]$. The process $(B_t)_{t\ge 0}$ and the random variables $\theta$ and $(U_k)_{k=1}^\infty$ are mutually independent. The law of $\theta$ is parameterized by $\pi\in[0,1]$ and we have $\P_\pi(\theta=0)=\pi$ and $\P_\pi(\theta>t|\theta>0)=\e^{-\lambda t}$ for $\lambda>0$. In keeping with the standard approach to quickest detection problems, for any $A\in\cF$
\[
\P_\pi(A)=\pi \P^0(A)+(1-\pi)\int_0^\infty \lambda \e^{-\lambda s}\P^s(A)\ud s,
\]
where $\P^t(A):=\P_\pi(A|\theta=t)$. Given a sub-$\sigma$-algebra $\cG\subset\cF$ and any two events $A,B\in\cF$ we denote $\P_\pi(A|\cG,B)=\P_\pi(A\cap B|\cG)/\P_\pi(B|\cG)$.

On this space we consider the stochastic dynamics
\begin{align*}
X_t = x +\mu(t-\theta)^+ + \sigma B_t, \quad\text{for } t\ge 0,
\end{align*}
where $x\in\R$, $\mu\in\R$, $\sigma>0$ and $(\,\cdot\,)^+$ denotes the positive part. The filtration generated by the sample paths of the process is denoted by $\cF^X_t=\sigma(X_s,\,s\le t)$, and it is augmented with $\P_\pi$-null sets.

Given random times $\rho\le \tau$, we denote open and closed intervals respectively by $(\!(\rho,\tau)\!)$ and $[\![\rho,\tau]\!]$. Given a filtration $(\cG_t)_{t\ge 0}$, we denote by $\cT(\cG_t)$ the class of stopping times for that filtration.
Let $\N$ be the set of natural numbers excluding zero. We say that a sequence $(\tau_n)_{n=0}^\infty$ of random times is {\em increasing} if:
\begin{itemize}
\item[(i)] $\tau_0=0$ and $\tau_n$ is $\P_\pi$-a.s.\ finite for all $n\in\N$,
\item[(ii)] $\tau_n\le \tau_{n+1}$ for all $n\in\N$ and
\begin{align}\label{eq:taun-lim}
\lim_{n\to\infty}\tau_n=\infty,\quad\text{$\P_\pi$-a.s}.
\end{align}
\end{itemize}
Given an increasing sequence $(\tau_n)_{n=0}^\infty$ we introduce random variables $(Z_n)_{n=0}^\infty$ which will be used to model the output of the tests that the optimizer runs on the system for the detection of $\theta$. More precisely, we set $Z_0=0$ and, for $k\in\N$ and $\epsilon\in[0,1)$, we define
\begin{align}\label{eq:Zk}
Z_k=
\left\{
\begin{array}{ll}
1, & \text { if $\theta\le\tau_{k}$ and $U_k\in[0, 1-\epsilon]$}, \\[+4pt]
0, & \text {  otherwise.}
\end{array}
\right.
\end{align}
It is important to notice that the sequence $(Z_n)_{n=0}^\infty$ depends on the increasing sequence $(\tau_n)_{n=0}^\infty$, so we should actually write  $Z_{n}=Z_n(\tau_n)$. However, we prefer to keep a simpler notation as no confusion shall arise.

Given an increasing sequence $(\tau_k)_{k=0}^\infty$ and the associated sequence $(Z_k)_{k=0}^\infty$, we define the processes
\[
C_t:=\sum_{i=0}^\infty 1_{\{t\ge \tau_i\}}\quad\text{and}\quad Y_t:=\sum_{i=0}^\infty Z_i 1_{\{t\ge \tau_i\}},
\]
and their natural filtration $\cF^{C,Y}_t=\sigma((C_s,Y_s),s\le t)$. The process $C$ is the counter of the number of tests, whereas the process $Y$ acts as a counter of the number of {\em positive} tests.  Letting $\cG^0_t=\cF^X_t$, we define recursively a  progressive  enlargement of the initial filtration by  setting
\[
\cG^{k}_t:=\cG^{k-1}_t\vee\sigma(Z_k 1_{\{s\ge \tau_k\}},s\le t)\vee\sigma(1_{\{s\ge \tau_k\}},s\le t),\quad\text{for $k\in\N$}.
\]
Notice that $(\cG^\infty_t)_{t\ge 0}$ is the filtration generated by the triplet $(C_t,X_t,Y_t)_{t\ge 0}$ with $\cG^\infty_t=\cF^X_t\vee\cF^{C,Y}_t$. Moreover, $\cG^k_t=\cF^X_t\vee\cF^{C,Y}_{t\wedge\tau_k}$ for every $k\in\N$.

By construction, each $\tau_k$ is a $(\cG^k_t)_{t\ge 0}$-stopping time. However, for the formulation of our optimization problem we further restrict the class of increasing sequences as explained in the next definition.

\begin{definition}[{\bf Admissible sequence}]\label{def:admis} We say that an increasing sequence $(\tau_n)_{n=0}^\infty$ of random times is {\em admissible}  if $\tau_k\in\cT(\cG^{k-1}_t)$ for each $k\in\N$. We denote by $\cT_*$ the class of admissible sequences of random times.

\end{definition}

Now we are ready to state the quickest detection problem we are interested in. Given an admissible sequence $(\tau_n)_{n=0}^\infty\in\cT_*$, the associated sequence $(Z_n)_{n=0}^\infty$ and the process $(Y_t)_{t\ge 0}$, we set
\begin{align}\label{eq:KY}
N_Y\big((\tau_n)_{n=0}^\infty\big)=N_Y:=\inf\{n\in\N : Y_{\tau_n}=1\}.
\end{align}
Here, $N_Y$ is the number of system inspections until the drift is detected and
\[
\tau_{N_Y}:=\inf\{t\ge 0:Y_t=1\}
\]
is the time at which the drift is detected. It is clear that on $\{N_Y=j\}$ we have $\tau_{N_Y}=\tau_j$. The cost function associated to a sequence $(\tau_n)_{n=0}^\infty\in\cT_*$ is defined as
\begin{align*}
\cJ_\pi\big((\tau_n)_{n=0}^\infty\big):=\E_\pi\Big[N_Y+\beta(\tau_{N_Y}-\theta)^+\Big],
\end{align*}
where $\beta>0$ is a given constant.
The interpretation is as follows: the optimizer pays one unit of capital each time she inspects the system; moreover, she also pays a penalty proportional to the delay with which the drift is detected. The problem ends upon drift detection.

The value function of the optimization problem reads
\begin{align}\label{eq:Vhat-TKY}
\hat V(\pi):= \inf_{(\tau_n)_{n=0}^\infty\in\cT_*}\cJ_\pi\big((\tau_n)_{n=0}^\infty\big).
\end{align}
In order to study this problem, our first step will be to produce a Bayesian formulation of the cost function. This will be done in the next subsection. Before moving on, some comments on the problem formulation are in order.

\begin{remark}[{\bf Intuition about problem formulation}]\label{rem:intuition}
\

\begin{itemize}
\item[(a)] Definition~\ref{def:admis} embodies the idea that along an admissible sequence $(\tau_n)_{n=0}^\infty$, the decision to stop and perform a test on the system can only depend on the observed trajectory of the signal $X$ and on the outputs of {\em past} tests.

\item[(b)] The processes $(C_t)_{t\ge 0}$ and $(Y_t)_{t\ge 0}$ are non-decreasing and constant on every random interval $[\![\tau_k,\tau_{k+1})\!)$. The process $(C_t)_{t\ge 0}$ increases by one unit at each stopping time in the admissible sequence. Instead, the process $(Y_t)_{t\ge 0}$ remains constant  and equal to zero until the first positive test and then again constant between any two subsequent positive tests. Clearly $Y_{\tau_j}\in\cG^{j}_{\tau_j}$ for each $j\in\{0\}\cup\N$.

\item[(c)] Notice that
$\{Z_k=0,\theta\le\tau_k\}=\{U_k>1-\epsilon,\theta\le\tau_k\}$. Since $U_k$ is uniformly distributed and independent of $\cG^{k-1}_\infty$  and $\theta$,  we have
\begin{equation}\label{eq:false.negative}
\P_\pi(Z_k=0\ |\ \cG^{k-1}_{\tau_k},\theta\le \tau_k)=\P_\pi(U_k>1-\epsilon)=\epsilon.
\end{equation}
The above shows that  $\epsilon\in[0,1)$ is precisely the probability that when the optimizer tries to detect the drift at the (stopping) time $\tau_k$, the test returns a false negative.
\end{itemize}

\end{remark}

\subsection{Bayesian formulation}\label{sec:bayesian}
The content of this section is motivated by the following simple observation: the cost function can be expressed as
\begin{align}\label{eq:payoff0}
\begin{aligned}
\E_\pi\Big[N_Y+\beta(\tau_{N_Y}-\theta)^+\Big]  &=\E_\pi\Big[\sum_{j=1}^{N_Y}\Big(1+\beta\int_{\tau_{j-1}}^{\tau_{j}}1_{\{\theta\le s\}}\ud s\Big)\Big] \\
&=\E_\pi\Big[\sum_{j=0}^{\infty}1_{\{Y_{\tau_j}=0\}}\Big(1+\beta\int_{\tau_{j}}^{\tau_{j+1}}1_{\{\theta\le s\}}\ud s\Big)\Big].
\end{aligned}
\end{align}
By definition of an admissible sequence $(\tau_j)_{j=0}^\infty$, we have $\{\tau_{j}\le s\le \tau_{j+1}\}\in\cG^{j}_{s} $ and $Y_{\tau_j}\in\cG^{j}_{\tau_j}$. Then, by tower property we have
\begin{align}\label{eq:sum}
\begin{aligned}
\E_\pi\Big[1_{\{Y_{\tau_j}=0\}}\int_{\tau_{j}}^{\tau_{j+1}}1_{\{\theta\le s\}}\ud s\Big]&=\E_\pi\Big[\int_{\tau_{j}}^{\tau_{j+1}}\P_\pi \big(\theta\le s,\, Y_{\tau_j}=0\big|\cG^j_{s}\big) \ud s\Big] \\
&=\E_\pi\Big[1_{\{Y_{\tau_j}=0\}}\int_{\tau_{j}}^{\tau_{j+1}}\Pi^j_s  \ud s\Big],
\end{aligned}
\end{align}
where in the final expression we set
\begin{align*}
\Pi_{s}^{j} = \P_\pi\big(\theta \leq s | \cG_s^{j}\big),\quad \text{for $s\in[\![\tau_j,\tau_{j+1})\!)$}.
\end{align*}
The latter process is the posterior process after the $j$th test has been performed on the system.  Plugging \eqref{eq:sum} into \eqref{eq:payoff0} we arrive at our Bayesian formulation of the problem:
\begin{align}\label{eq:Vhat}
\hat V(\pi)=\inf_{(\tau_n)_{n=0}^\infty\in\cT_*}\E_\pi\Big[\sum_{j=0}^\infty1_{\{Y_{\tau_j}=0\}}\Big(1+\beta\int_{\tau_{j}}^{\tau_{j+1}}\Pi^j_s\ud s\Big)\Big].
\end{align}

In order to be able to solve the problem in its Bayesian formulation, we need more explicit characterization of the dynamics of the posterior processes $(\Pi^k_t)_{t\ge 0}$ for all $k\in\N\cup\{0\}$.
Let us define the family of posteriors $\{\Pi^0_\tau,\,\tau\in\cT(\cG^0_t)\}$, where $\Pi_\tau^0 :=  \P_\pi (\theta \leq \tau  |  \cG^0_\tau )$.
This family forms a so-called $T$-sub-martingale system \cite[Ch.\ 2.11]{elkaroui1981} and it is easy to check that $t\mapsto \E_\pi[\Pi_t^0]$ is right-continuous. Therefore the family can be aggregated in a c\`adl\`ag sub-martingale process \cite[Prop.\ 2.14]{elkaroui1981}, which we still denote by $(\Pi_t^0)_{t\ge 0}$ with a slight abuse of notation.
It is actually well known that $\Pi_t^0$ satisfies the stochastic differential equation (SDE) \cite[Sec.\ 22]{peskir2006optimal}
\begin{align}\label{eq:SDEpi0}
\Pi^0_t =\pi+\int_0^t \lambda (1 - \Pi^0_s) \ud s + \frac{\mu}{\sigma}
\int_0^t\Pi^0_s (1 - \Pi^0_s) \ud W_s,
\end{align}
where
\begin{align}\label{eq:W}
W_t=\frac{1}{\sigma}\Big(X_t-\mu\int_0^t\Pi^0_s\ud s\Big)
\end{align}
is a standard Brownian motion with respect to $(\cF^X_t,\P_\pi)$.

Next, we are going to consider the families of posteriors $\{\Pi^k_\tau,\,\tau\in\cT(\cG^k_t)\}$, where $\Pi_\tau^k :=  \P_\pi (\theta \leq \tau  |  \cG^k_\tau )$, for any $k\in\N$. By the same arguments as above we can aggregate each family into a stochastic process $(\Pi^k_t)_{t\ge 0}$.  Clearly $(\Pi^k_{t\vee\tau_k})_{t\ge 0}$ is a $(\cG^k_{t\vee\tau_k})_{t\ge 0}$-adapted process and $1_{\{Y_{\tau_k}>0\}}\Pi_t^k=1_{\{Y_{\tau_k}>0\}}$ for $t \in [\![\tau_{k}, \infty)\!)$. Mimicking arguments from \cite[Sec.\ 22]{peskir2006optimal}, we obtain the dynamics of the posterior $\Pi^k$. The result is not entirely trivial due to the adjustment of the posterior process following each attempt to detect the drift. In order to account for the latter feature, we must introduce \begin{align}\label{eq:def-g}
g(\pi):=\frac{\epsilon \pi}{1-(1-\epsilon)\pi},\quad\text{for } \pi\in[0,1].
\end{align}
When necessary we write $g=g_\epsilon$ to emphasise the dependence on $\epsilon$.
\begin{proposition}\label{prop:dynamics}
Let $(\tau_k)_{k=0}^\infty$ be an admissible sequence of stopping times. Then, for $k\in\N$ and $t\ge 0$,
\begin{align}\label{eq:Pik}
\Pi^k_{\tau_k+t} =\Pi^{k}_{\tau_k}+\int_{\tau_k}^{\tau_k+t} \lambda (1 - \Pi^k_s) \ud s +  \frac{\mu}{\sigma} \int_{\tau_k}^{\tau_k+t}\Pi^k_s (1 - \Pi^k_s) \ud W^k_s,
\end{align}
with
\[
\Pi^{k}_{\tau_k}=g(\Pi^{k-1}_{\tau_k})1_{\{Y_{\tau_k}=0\}}+1_{\{Y_{\tau_k}>0\}},
\]
and where, for $s\ge 0$
\[
W^k_{\tau_k+s}-W^k_{\tau_k}=\frac{1}{\sigma}\int_{\tau_k}^{\tau_k+s} \big(\ud X_u-\mu\Pi^k_u\ud u\big)
\]
is a $((\cG^{k}_{\tau_k+s})_{s\ge 0},\P_\pi)$-Brownian motion.
\end{proposition}
\begin{proof}
See Section~\ref{sec:proof.dynamics} in Appendix.
\end{proof}

\begin{remark}We make some observations concerning the dynamics of $\Pi^k$.
\begin{itemize}
\item[(a)] It is shown in the proof of the proposition above that (cf.\ \eqref{eq:gmeaning})
\[
\P_\pi( \theta \leq \tau_k |\cG^{k-1}_{\tau_k}, Z_k=0)=g_\epsilon\big(\Pi^{k-1}_{\tau_k}\big).
\]
Therefore, the function $g_\epsilon(\pi)$ represents the posterior probability that a drift has actually appeared, despite the observation of a negative outcome from a test with prior equal to $\pi$.
\item[(b)] Notice that $0<g_\epsilon(\pi)<\pi$ for every $\epsilon\in(0,1)$,  meaning that the posterior makes a downward jump after each negative test. The size of the jump increases with decreasing $\epsilon$. That is, when there is a small probability of a false negative, a negative test is very informative and it produces a large downward shift in the posterior process. Vice versa, when the probability of a false negative is large, a negative test does not provide much information and it does not significantly change the posterior.
\item[(c)] Recall that $\cG^\infty_t=\cF^X_t\vee\cF^{C,Y}_{t}$ and $\cG^k_t=\cF^X_t\vee\cF^{C,Y}_{t\wedge\tau_k}$ for $t\in[0,\infty)$, for every $k\in\N\cup\{0\}$. Notice also that $t\mapsto\cF^{C,Y}_t$ is constant for $t\in[\![\tau_k,\tau_{k+1})\!)$. Setting
\begin{align}\label{eq:Pi}
\Pi_{t} = \P_\pi\big(\theta \leq t | \cG^\infty_t\big), \quad \text{for $t \in [0, \infty)$},
\end{align}
we have $\Pi_t = \Pi_t^k$ for $t\in[\![\tau_k,\tau_{k+1})\!)$. So the process $(\Pi_t)_{t\ge 0}$ summarizes in a compact form the collection of all posteriors $\{(\Pi^k_t)_{t\ge 0},k\in\N\cup\{0\}\}$,  but in this representation we do not see explicitly the admissible sequence $(\tau_k)_{k=0}^\infty$.
\end{itemize}
\end{remark}

\begin{figure}
\centerline{\includegraphics[width=0.85\linewidth]{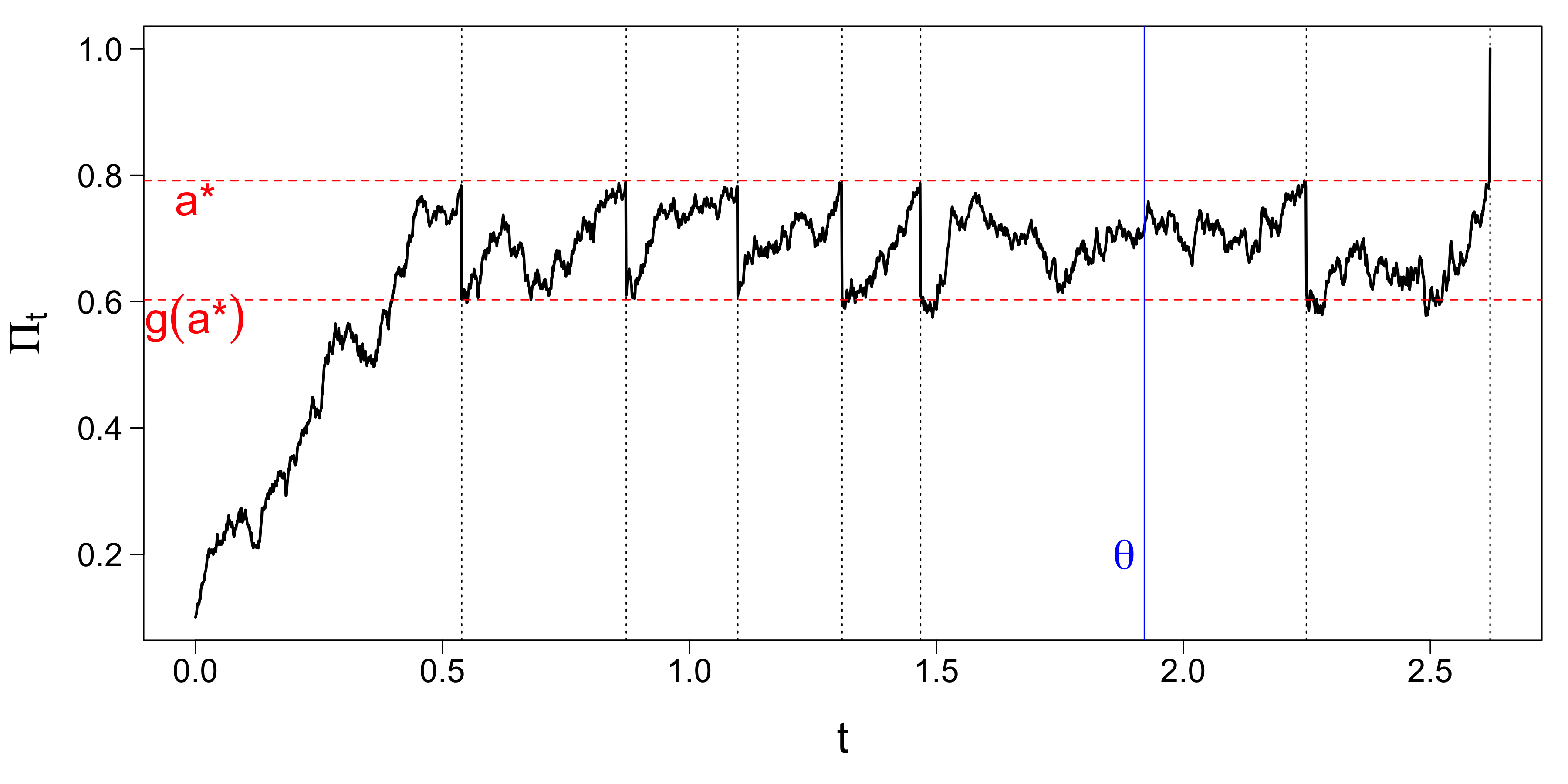}}
\caption{A numerical example of the optimally stopped posterior probability process
    for parameters set to $\lambda = 2$, $\beta = 1.5$, $\mu=1$, $\sigma = 1$, $\epsilon = 0.4$, $\pi = 0.1$.
    The optimizer performs a test on the system whenever the posterior probability hits $a^* = 0.792$.  If the test returns a negative outcome the new posterior becomes $g(a^*)=0.603$. For this simulated example, $\theta = 1.92$ (indicated by the straight vertical line), and in total the observer performs 7 tests (indicated by the vertical dotted lines). In particular, the drift is detected after one false negative. It may be worth noticing that the realization of $\theta$ is about 3 standard deviations away from the mean. That is why several inspections happen before the drift has actually appeared.
    } \label{fig:path}
\end{figure}

\subsection{Main results}\label{sec:mainresults}
The problem in \eqref{eq:Vhat} can be read as an optimal multiple stopping problem of a kind which formally fits within the framework of \cite{kobylanski2011optimal}.
For the reader's convenience, we summarize in this section the main results of the paper and we will later provide a proof in a number of steps presented in subsequent sections.

In order to give a formal statement we need to introduce some notation. Let $\gamma=\mu^2/(2\sigma^2)$ and $\rho =  \lambda /\gamma$, and define the functions
\begin{align}
h(\pi) &=  \rho \left\{ \log \frac{\pi}{1 - \pi}  - \frac{1}{\pi} \right\},\label{eq:def.h}\\
\psi(\pi) &=- \frac{\beta}{\gamma} \e^{- h (\pi)} \int_0^\pi  \frac{   \e^{    h (x)   }}{   x (1 - x)^2 }\ud x, \label{eq:def.psi}
\end{align}
and
\begin{equation}\label{eq:def-Psi}
    \Psi(\pi)=\int_0^\pi\psi(x)\ud x.
\end{equation}
See Remark~\ref{rmk:hitting-Psi} for the interpretation of $\Psi$.
For any $\epsilon \in (0, 1)$, the inverse function $g^{-1} = g^{-1}_\epsilon$ is given by
\begin{equation} \label{eq:def.ginv}
g^{-1}_\epsilon (q) =  \frac{q}{\epsilon +  q(1 - \epsilon)},  \quad \text{ for } q \in [0, 1].
\end{equation}
When $\epsilon = 0$, we define $g^{-1}_0 (q) = \lim_{\epsilon \downarrow 0} g^{-1}_\epsilon(q)$; that is $g^{-1}_0(q) = 0$ if $q = 0$ and $g^{-1}_0(q) = 1$ otherwise.
Clearly, $g(\pi)<\pi$ for $\pi\in(0,1)$ whereas  $g^{-1}(q)> q$ for $q\in(0,1)$. For $k\in\N\cup\{0\}$, let $g^k(q) = (g \circ \ldots \circ g)(q)$ and $g^{-k}(q)=(g^{-1}\circ\ldots\circ g^{-1})(q)$ denote the $k$-fold composition of $g$ and $g^{-1}$ respectively with the convention $g^0(q):=q$ for $q\in(0,1)$. For any $q\in(0,1)$, we will show in Remark~\ref{rem:g-1} that  $g^{-k}(q)<g^{-(k+1)}(q)$ for all $k\in\N\cup\{0\}$ with $g^{-k}(q)\uparrow 1$ as $k\to\infty$.
Finally, for bounded measurable $v:[0,1]\to(0,\infty)$ we denote
\begin{align}\label{eq:cA}
(\cA v)(\pi) := 1 + [1 - (1 - \epsilon) \pi ] v (  g(\pi) )
\end{align}
and we let $\cA^k v=[\cA\circ\ldots\circ\cA]v$ be the $k$-fold application of the operator $\cA$ to $v$.

\begin{theorem}\label{thm:main}
The value function $\hat V$ of the quickest detection problem reads
\begin{align}\label{eq:hatVsol}
\hat V(\pi)=
\left\{
\begin{array}{ll}
\Psi(\pi)+C, & \pi\in[0,a^*),\smallskip\\
\big[\cA^k (\Psi + C)\big] (\pi), & \pi\in [g^{-(k-1)}(a^*),g^{-k}(a^*)),\, k\in\N,
\end{array}
\right.
\end{align}
where the constant $C$ is given by
\begin{equation}
    C = \frac{1}{(1-\epsilon) a^*} + \frac{1 - (1 - \epsilon)a^* }{ (1 - \epsilon)a^* }\Psi( g(a^*) )  - \frac{1}{(1 - \epsilon) a^*} \Psi(a^*),
\end{equation}
and $a^*$ is the unique solution in $(0,1)$ of
\begin{align}
 1 + \Psi(g(a))  - g(a) \psi(g(a))   - \Psi(a)   + a \psi(a) = 0.
\end{align}

The sequence $(\tau^*_n)_{n=0}^\infty$ given by:
\begin{align*}
\tau^*_0=0\quad\text{and}\quad\tau^*_{n}=\bar\tau_n 1_{\{Y_{\tau^*_{n-1}}=0\}}+(\tau^*_{n-1}+1) 1_{\{Y_{\tau^*_{n-1}}>0\}},\quad\text{for $n\in\N$,}
\end{align*}
with
$\bar \tau_{n} = \inf\{t \geq  \tau^*_{n-1}:  \Pi^{n-1}_{t}\ge a^*\}$ is admissible and optimal for $\hat V$, in the sense that $\hat V(\pi)=\cJ_\pi((\tau^*_n)_{n=0}^\infty)$ for all $\pi\in[0,1]$.
\end{theorem}

The theorem gives an explicit formula for the value function of our problem (up to determining $a^*$ as the unique root of an algebraic equation) and an optimal strategy for the stopper. It is optimal to perform a test on the system each time the posterior process reaches the upper boundary $a^*$. If the test returns a negative value, the posterior is adjusted to the new value $g(a^*)$ according to the dynamics in Proposition~\ref{prop:dynamics}. The drift is detected at $\tau^*_{N_Y}$, and the subsequent stopping times are simply chosen after each time unit. The latter choice is arbitrary, and it is only needed to guarantee $\tau^*_n\to\infty$ as $n\to\infty$. Using the dynamics in \eqref{eq:Pik}, in Fig.~\ref{fig:path} we give a numerical example for the process $(\Pi_t)_{t\ge 0}$ defined in \eqref{eq:Pi} when the optimizer implements the optimal strategy from the theorem above.

The proof of the theorem is obtained in several steps in the following three sections and it is formally presented in the proof of Theorem~\ref{cor:V}. In the process we also obtain continuous differentiability of the value function $\hat V$ on $(0,1)$ and its continuity on $[0,1]$. These properties are also stated in Theorem~\ref{cor:V}.

The strategy of proof relies on two main steps: first, in the next section we show that the problem in \eqref{eq:Vhat-TKY} is equivalent to a recursive optimal stopping problem (cf.\ Proposition~\ref{prop:recur}); then, in Section~\ref{sec:solution} we solve the recursive problem with free boundary methods. Finally, in Theorem~\ref{cor:V} we obtain the results stated in the theorem above.

We end this section with two useful observations.
\begin{remark} \label{rem:g-1}
For any $\epsilon \in [0, 1)$ and  $a \in (0, 1)$ let $J_0=[0,a)$ and $J_k=[g^{-(k-1)}(a),g^{-k}(a))$. We observe that $\cup_{k=0}^\infty J_k=[0,1)$. If $\epsilon = 0$, the claim is trivial, since we have $J_1 = [a, 1)$.
If $\epsilon \in (0, 1)$, we need to show that $g^{-k}(a) \rightarrow 1$ as $k \rightarrow \infty$.
Since $g^{-1} (q) \in (q, 1)$ for any $q \in (0, 1)$, the sequence $(g^{-k}(a))_{k\in\N}$ is increasing. Then, $L  := \lim_{k \rightarrow \infty} g^{-k}(a)$ exists and it must satisfy $L  = g^{-1}(L)$. Solving the equation we get $L = 1$.

\end{remark}
\begin{remark}\label{rmk:hitting-Psi}
Let $\cL$ denote the infinitesimal generator of the process $\Pi_t^0$, which has dynamics given by~\eqref{eq:SDEpi0}. That is,
\begin{equation}\label{eq:inf-generator}
    (\cL f)(\pi):=\lambda(1-\pi)f'(\pi)+\gamma\pi^2(1-\pi)^2f''(\pi),
\end{equation}
where $\gamma=\mu^2/ (2 \sigma^2)$ and $f$ is a sufficiently smooth function.
Let $\Pi_t^0$ have initial value $\Pi_0^0 = \pi$.
Since the function $\Psi$ given by~\eqref{eq:def-Psi} satisfies  $(\cL\Psi)(\pi)+\beta\pi=0$, a routine argument using It\^{o}'s lemma shows that, for any $ a \in [\pi, 1)$,
\begin{equation}\label{eq:hitting-Psi}
  \Psi(\pi) - \Psi(a)
  = \E \Big[\beta\int_0^{\tau_a} \Pi_t^0 \ud t\Big],
\end{equation}
where  $\tau_a=\inf\{t\ge 0 \colon \Pi_t^0\ge a\}$.
This provides an alternative approach to quickly computing the value function $\hat{V}$ for the optimal strategy.  Assume $\pi < a^*$. By the calculations in \eqref{eq:YkGk} we have, for $j\in\N$,
\begin{align*}
\P_\pi(Y_{\tau^*_j}=0)&=\E_\pi\Big[1_{\{Y_{\tau^*_{j-1}}=0\}}\P_\pi\big(Z_j=0\big|\cG^{j-1}_{\tau^*_j}\big)\Big]\\
&=\E_\pi\Big[1_{\{Y_{\tau^*_{j-1}}=0\}}\big(1\!-\!(1\!-\!\epsilon)\Pi^{j-1}_{\tau^*_j}\big)\Big]=\big(1\!-\!(1\!-\!\epsilon)a^*\big)\P_\pi\big(Y_{\tau^*_{j-1}}=0\big).
\end{align*}
Iterating the argument we deduce $\P_\pi(Y_{\tau^*_j}=0)=(1\!-\!(1\!-\!\epsilon)a^*)^j$, i.e., at each optimal stopping time $\tau_j^*$  the optimization ends with probability $(1 - \epsilon) a^*$. A routine calculation using \eqref{eq:hitting-Psi} and tower property yields that, for $\pi<a^*$,
\begin{align}\label{eq:hatV-alter}
\begin{aligned}
    \hat{V}(\pi) &=  \sum_{j=0}^\infty \P_\pi\big(Y_{\tau^*_j}=0 \big)
+ \E_\pi \Big[   \beta\int_0^{\tau^*_{1}}\Pi^0_s\ud s \Big]+ \E_\pi \Big[ \sum_{j=1}^\infty 1_{\{Y_{\tau^*_j}=0\}}  \beta \int_{\tau^*_{j}}^{\tau^*_{j+1}}\Pi^j_s\ud s \Big] \\
&=  \frac{1}{(1-\epsilon) a^*} + \left\{ \Psi(\pi) - \Psi(a^*) \right\}+ \E_\pi \Big[ \sum_{j=1}^\infty 1_{\{Y_{\tau^*_j}=0\}}  \E_\pi\Big[\beta \int_{\tau^*_{j}}^{\tau^*_{j+1}}\Pi^j_s\ud s\Big|\cG^\infty_{\tau^*_j} \Big] \Big]\\
&=  \frac{1}{(1-\epsilon) a^*} + \left\{ \Psi(\pi) - \Psi(a^*) \right\} +  \frac{1 - (1 - \epsilon)a^* }{ (1 - \epsilon)a^* }
 \left\{ \Psi( g(a^*) )  -  \Psi(a^*)\right\},
 \end{aligned}
\end{align}
where we used that on $\{Y_{\tau^*_j}=0\}$ it holds
\[
\E_\pi\Big[\beta \int_{\tau^*_{j}}^{\tau^*_{j+1}}\Pi^j_s\ud s\Big|\cG^\infty_{\tau^*_j} \Big]=\E_{g(a^*)}\Big[\beta\int_0^{\tau_{a^*}}\Pi^0_t\ud t\Big].
\]
The right-hand side of \eqref{eq:hatV-alter} coincides with the expression given in Theorem~\ref{thm:main}.
In particular, one needs on average $[(1-\epsilon)a^*]^{-1}$ inspections to eventually detect the drift.
\end{remark}

\section{A recursive problem and its equivalence with the quickest detection}\label{sec:recursive}

In this section we introduce a recursive problem and we prove its equivalence to the Bayesian formulation  in \eqref{eq:Vhat}. The recursive problem is stated here:
\medskip

\noindent{\bf Problem [R].}
{\em Let $(\bar \Omega,\bar \cF,\bar \P)$ be a probability space equipped with a Brownian motion $(\bar W_t)_{t\ge 0}$ and let $\bar \Pi_0$ be a random variable independent of $\bar W$. Let $(\bar \Pi_t)_{t\ge 0}$ be the solution of
\begin{align}\label{eq:Pibar}
\bar \Pi_t =\bar \Pi_0+\int_{0}^t \lambda (1 - \bar \Pi_s) \ud s + \frac{\mu}{\sigma} \int_{0}^t\bar \Pi_s (1 - \bar \Pi_s) \ud \bar W_s,
\end{align}
and denote its natural filtration by $(\bar \cF_t)_{t\ge 0}$ (augmented with the $\bar \P$-null sets). Find a bounded measurable function $V: [0, 1] \rightarrow (0, \infty)$ which satisfies the recursive formula
\begin{equation}\label{eq:def.v}
V(\pi) = \inf_{\tau\in\cT(\bar \cF_t)} \bar \E_\pi \Big[\beta\int_0^\tau \bar \Pi_t\ud t  + \cA V \big(\bar \Pi_\tau\big)\Big],
\end{equation}
where
$\cT(\bar\cF_t)$ is the class of a.s.\ finite stopping times for the filtration $(\bar \cF_t)_{t\ge 0}$ and
$\bar\E_\pi[\cdot]$ is the expectation under the measure $\bar \P_\pi(\,\cdot\,)=\bar \P(\,\cdot\,|\bar\Pi_0=\pi)$. \hfill$\square$}

\medskip

Now we prove existence of a solution of {\bf Problem [R]}.

\begin{proposition}There exists a solution of {\bf Problem [R]}.

\end{proposition}
\begin{proof}
Let us start by assuming that we can construct a sequence $(w_n)_{n\in\N}$ for functions $[0,1]\to[0,\infty)$ such that $w_1$ is bounded, $w_n\ge w_{n+1}\ge 0$ for all $n\in\N$ and
\begin{align}\label{eq:WN}
w_{n+1}(\pi):=\inf_{\sigma\in\cT(\bar \cF_t)}\bar \E_\pi\Big[1+\beta\int_{0}^{\sigma}\bar \Pi_t\ud t+\big(1\!-\!(1\!-\!\epsilon)\bar \Pi_{\sigma}\big)w_{n}\big(g(\bar\Pi_{\sigma})\big)\Big],
\end{align}
for $\pi\in[0,1]$. Then we can define the limit
\[
w_\infty(\pi):=\lim_{n\to\infty}w_n(\pi),\quad \pi\in[0,1].
\]
Since $w_{n}\ge w_\infty$ for $n\in\N$, then
\begin{align*}
w_{n+1}(\pi)\ge \inf_{\sigma\in\cT(\bar \cF_t)}\bar \E_\pi\Big[1+\beta\int_{0}^{\sigma}\bar \Pi_t\ud t+\big(1\!-\!(1\!-\!\epsilon)\bar \Pi_{\sigma}\big)w_{\infty}\big(g(\bar\Pi_{\sigma})\big)\Big].
\end{align*}
Letting $n\to\infty$, the latter yields
\begin{align*}
w_\infty(\pi)\ge \inf_{\sigma\in\cT(\bar \cF_t)}\bar \E_\pi\Big[1+\beta\int_{0}^{\sigma}\bar \Pi_t\ud t+\big(1\!-\!(1\!-\!\epsilon)\bar \Pi_{\sigma}\big)w_{\infty}\big(g(\bar\Pi_{\sigma})\big)\Big].
\end{align*}
For the reverse inequality we notice that, for any $\delta>0$ there is $\sigma_\delta\in\cT(\bar \cF_t)$ such that
\begin{equation}
\begin{aligned}
\inf_{\sigma\in\cT(\bar\cF_t)}&\bar \E_\pi\Big[1+\beta\int_{0}^{\sigma}\bar \Pi_t\ud t+\big(1\!-\!(1\!-\!\epsilon)\bar \Pi_{\sigma}\big)w_{\infty}\big(g(\bar\Pi_{\sigma})\big)\Big]\\
\ge& \bar \E_\pi\Big[1+\beta\int_{0}^{\sigma_\delta}\bar \Pi_t\ud t+\big(1\!-\!(1\!-\!\epsilon)\bar \Pi_{\sigma_\delta}\big)w_{\infty}\big(g(\bar\Pi_{\sigma_\delta})\big)\Big]-\delta\\
=&\bar \E_\pi\Big[1+\beta\int_{0}^{\sigma_\delta}\bar \Pi_t\ud t+\big(1\!-\!(1\!-\!\epsilon)\bar \Pi_{\sigma_\delta}\big)\lim_{n\to\infty}w_{n-1}\big(g(\bar\Pi_{\sigma_\delta})\big)\Big]-\delta\\
=&\lim_{n\to\infty}\bar \E_\pi\Big[1+\beta\int_{0}^{\sigma_\delta}\bar \Pi_t\ud t+\big(1\!-\!(1\!-\!\epsilon)\bar \Pi_{\sigma_\delta}\big)w_{n-1}\big(g(\bar\Pi_{\sigma_\delta})\big)\Big]-\delta\\
\ge &\lim_{n\to\infty} w_n(\pi)-\delta=w_\infty(\pi)-\delta,
\end{aligned}
\end{equation}
where the second equality is by the monotone convergence theorem.
Now, by arbitrariness of $\delta$ we conclude that $w_\infty$ is a solution of {\bf Problem [R]}.

It remains to construct the sequence $(w_n)_{n\in\N}$ with the properties above. First of all, we notice that if $V$ is a solution to {\bf Problem [R]}, then it must be $V(1)=1+\epsilon V(1)$, hence $V(1)=(1-\epsilon)^{-1}$. With this insight we set
\[
w_n(1)=(1-\epsilon)^{-1},\quad\text{for all $n\in\N$}.
\]
It suffices that $w_1\ge 0$ and $w_2\le w_1$, for then $0\le w_3\le w_2$ by \eqref{eq:WN} and we conclude by induction that $0\le w_{n+1}\le w_n$ for all $n\in\N$. In order to construct a suitable $w_1$ let us
fix an arbitrary $a\in(0,1)$ and define
\[
\varphi_a(\pi)=\varphi(\pi):=\bar \E_\pi\Big[\beta\int_0^{\tau_a}\bar \Pi_t\ud t\Big],
\]
with $\tau_a=\inf\{t\ge 0:\bar \Pi_t\ge a\}$.
As shown in~\eqref{eq:hitting-Psi}, $\varphi$ can be expressed by
\begin{align*}
\varphi(\pi)=\left\{
\begin{array}{ll}
\Psi(\pi) - \Psi(a), & \pi\in[0,a),\\
0,& \pi\in[a,1],
\end{array}
\right.
\end{align*}
where we recall the function $\psi$ defined in \eqref{eq:def.psi}.
Next, let us set
\[
c(a):=\frac{1+\varphi\big(g(a)\big)}{(1-\epsilon)a},
\]
and let us define $w_1$ piecewise as follows. For $\pi\in[0,a]$, we let
\begin{equation*}
w_1(\pi)=1+\varphi(\pi)+\big(1-(1-\epsilon)a\big)c(a).
\end{equation*}
Notice that $w_1(g(a))=c(a)$. Recall $g^{-1}$ from \eqref{eq:def.ginv} and, for $\pi\in(a,g^{-1}(a)]$, let
\begin{equation*}
w_1(\pi)=1+\big(1-(1-\epsilon)\pi\big)w_1\big(g(\pi)\big).
\end{equation*}
Then $w_1$ is continuous at $a$. We continue recursively and define
\begin{equation}\label{eq:def-w1}
    w_1(\pi)=1+\big(1-(1-\epsilon)\pi\big)w_1\big(g(\pi)\big)
\end{equation}
on each interval $(g^{-k}(a),g^{-(k+1)}(a)]$, for $k\in\N$. By this construction we have that $w_1$ is continuous at all points $g^{-k}(a)$, $k\in\N$. That can be proven by induction. Indeed, since $w_1$ is continuous at $a=g^0(a)$, assuming that $w_1$ is continuous at $g^{-k}(a)$, we have
\begin{align*}
\lim_{\pi\uparrow g^{-(k+1)}(a)}w_1(\pi)&=1+\big(1-(1-\epsilon)g^{-(k+1)}(a)\big)\lim_{\pi\uparrow g^{-(k+1)}(a)}w_1\big(g(\pi)\big)\\
&=1+\big(1-(1-\epsilon)g^{-(k+1)}(a)\big)\lim_{\pi\uparrow g^{-k}(a)}w_1\big(\pi\big)\\
&=1+\big(1-(1-\epsilon)g^{-(k+1)}(a)\big)\lim_{\pi\downarrow g^{-k}(a)}w_1\big(\pi\big)\\
&=1+\big(1-(1-\epsilon)g^{-(k+1)}(a)\big)\lim_{\pi\downarrow g^{-(k+1)}(a)}w_1\big(g(\pi)\big)\\
&= \lim_{\pi\downarrow g^{-(k+1)}(a)}w_1(\pi),
\end{align*}
where the induction assumption is used in the third equality, and the other equalities follow from the continuity of $g$ and $g^{-1}$ and the definition of $w_1$ given in~\eqref{eq:def-w1}.
In summary, $w_1$ can be expressed by
\begin{align*}
w_1(\pi)=\left\{\begin{array}{ll}
1+\varphi(\pi)+\big(1-(1-\epsilon)a\big)c(a), & \text{for $\pi\in[0,a]$},\\
1+\big(1-(1-\epsilon)\pi\big)w_1\big(g(\pi)\big), &\text{for $\pi\in(a,1)$},\\
(1-\epsilon)^{-1}, &\text{for $\pi=1$},
\end{array}
\right.
\end{align*}
and we note that it is continuous everywhere including at $\pi = 1$. The latter also implies that $w_1$ is bounded on $[0,1]$ as needed.
Now, for $\pi\in[0,a]$, choosing $\sigma=\tau_a$ we obtain
\begin{align*}
w_{2}(\pi)=&\inf_{\sigma\in\cT(\bar \cF_t)}\bar \E_\pi\Big[1+\beta\int_{0}^{\sigma}\bar \Pi_t\ud t+\big(1\!-\!(1\!-\!\epsilon)\bar \Pi_{\sigma}\big)w_{1}\big(g(\bar\Pi_{\sigma})\big)\Big]\\
&\le 1+\varphi(\pi)+\big(1\!-\!(1\!-\!\epsilon)a\big)w_{1}\big(g(a)\big)\\
&= 1+\varphi(\pi)+\big(1\!-\!(1\!-\!\epsilon)a\big)c(a)=w_1(\pi).
\end{align*}
Instead, for $\pi\in[a,1)$, choosing $\sigma=0$ we obtain
\[
w_2(\pi)\le 1+\big(1\!-\!(1\!-\!\epsilon)\pi\big)w_{1}\big(g(\pi)\big)=w_1(\pi).
\]
Hence $w_2(\pi)\le w_1(\pi)$ for $\pi\in[0,1)$ and $w_1(1)=w_2(1)=(1-\epsilon)^{-1}$, as required.

Since $w_1\ge 0$ is bounded and $0\le w_2\le w_1$, then $w_n\ge w_{n+1}\ge 0$ and the limit $w_\infty$ is bounded and non-negative.
\end{proof}

Next we establish the equivalence between the recursive problem and the Bayesian quickest detection problem.
\begin{proposition}\label{prop:recur}
Assume $V$ is a solution to {\bf Problem [R]}. Then $V(\pi) \le \hat V(\pi)$ for any $\pi \in[0, 1]$ with $\hat V$ as in \eqref{eq:Vhat}.

Assume in addition that
\begin{align}\label{eq:tauV}
\tau^*_V  = \inf\{t \geq 0 :  V(\bar\Pi_{t}) = (\cA V)(\bar\Pi_{t})\}
\end{align}
is optimal in {\bf Problem [R]} and the sequence of stopping times $(\tau^*_n)_{n\in\N}$ defined below is admissible:
set $\tau^*_0=0$ and, for $n\in\N$,
\begin{align}\label{eq:taun}
&\tau^*_n = \bar\tau_n 1_{\{Y_{\tau^*_{n-1}}=0\}}+(\tau^*_{n-1}+1)1_{\{Y_{\tau^*_{n-1}}>0\}}\,,
\end{align}
where
\begin{align}\label{eq:bartau}
\bar \tau_{n} = \inf\{t \geq  \tau_{n-1}^*:  V(\Pi^{n-1}_{t}) = (\cA V)(\Pi_{t}^{n-1})\}.
\end{align}
Then $V(\pi) = \hat V(\pi)$ for any $\pi \in [0, 1]$ and $(\tau^*_n)_{n=0}^\infty$ is optimal in \eqref{eq:Vhat}.
\end{proposition}
\begin{proof}
Fix $\pi\in[0,1]$. It is clear that if $\bar \Pi_0=\pi$ then
\begin{align}\label{eq:law0}
\mathsf{Law}\big((\bar \Pi_t)_{t\ge 0}|\bar \P\big)=\mathsf{Law}\big((\Pi^0_t)_{t\ge 0}|\P_\pi\big).
\end{align}
Optimality of $\tau^*_V$ implies that the law of the underlying process completely determines the value function and therefore
\begin{align}\label{eq:VPi0}
V(\pi)=\inf_{\tau\in\cT(\cG^0_t)}\E_\pi\Big[\beta\int_0^\tau\Pi^0_t\ud t+(\cA V)(\Pi^0_\tau)\Big].
\end{align}

Take $\tau_1$ from an admissible sequence $(\tau_n)_{n=0}^\infty\in\cT_*$. It follows from \eqref{eq:VPi0} that
\begin{align}\label{eq:v1}
V(\pi) &\leq \E_\pi \Big[ 1 + \beta \int_{0}^{\tau_1} \Pi^0_t\ud t  + (1 - \Pi^0_{\tau_1} + \epsilon \Pi^0_{\tau_{1}}) V\big(g(\Pi^0_{\tau_{1}}
)\big) \Big].
\end{align}
For $Z_1$ defined as in \eqref{eq:Zk} we have
\begin{equation*}
\begin{aligned}
\E_\pi\big[1_{\{Z_{1} = 0 \}}V\big(g(\Pi^0_{\tau_{1}})\big) \big|  \cG_{\tau_1}^0\big]  &=  V\big(g(\Pi^0_{\tau_{1}})\big)\P_\pi\big(Z_{1} =  0 \big| \cG_{\tau_1}^0 \big) = V\big(g(\Pi^0_{\tau_{1}})\big)\big(1 - (1 - \epsilon) \Pi^0_{\tau_{1}}\big),
\end{aligned}
\end{equation*}
where the final expression holds because (cf.\ Remark~\ref{rem:intuition}-(c))
\[
\P_\pi\big(Z_{1} =  0 \big| \cG_{\tau_1}^0 \big) = \P_\pi\big(Z_{1} =  0,\theta\le \tau_1 \big| \cG_{\tau_1}^0 \big) +\P_\pi\big(\theta>\tau_1 \big| \cG_{\tau_1}^0 \big) =\epsilon\Pi^0_{\tau_1}+1-\Pi^0_{\tau_1}.
\]
Then, using tower property we can write
\begin{equation}\label{eq:mu01}
\begin{aligned}
\E_\pi & \Big[(1 - \Pi^0_{\tau_1} + \epsilon \Pi^0_{\tau_{1}}) V\big(g(\Pi^0_{\tau_{1}}
)\big)\Big] = \E_\pi\Big[1_{\{Z_1=0\}} V\big(g(\Pi^0_{\tau_{1}}
)\big)\Big] = \E_\pi\Big[1_{\{Y_{\tau_1}=0\}} V\big(\Pi^1_{\tau_{1}}
\big)\Big],
\end{aligned}
\end{equation}
where we used that $\Pi^1_{\tau_{1}}=g(\Pi^0_{\tau_{1}})$ on the event $\{Z_1=0\}=\{Y_{\tau_1}=0\}$.

For $k\in\N$ we set
\begin{align}\label{eq:tildev}
\widetilde V^k(\Pi^k_{\tau_k})=\essinf_{\substack{\tau\in\cT(\cG^k_t)
\\
\tau\ge\tau_k}} \E_\pi \Big[\beta\int_{\tau_k}^\tau \Pi^k_t\ud t  + (\cA V) \big( \Pi^k_\tau\big)\Big|\cG^k_{\tau_k} \Big].
\end{align}
The identity
$V(\pi)=\widetilde V^k(\pi)$ holds for all $\pi\in[0,1]$; although this fact is somewhat well known, we provide a proof in Section~\ref{app:VV} of the Appendix for completeness. Since $\widetilde V^k=V$ and $\tau^*_V$ is optimal by assumption, on the event $\{Y_{\tau_k}=0\}$ the stopping time
\begin{align}\label{eq:tautilde}
\widetilde\tau_{k+1}^*:=\inf\{t\ge \tau_k: V(\Pi^k_t)=(\cA V)(\Pi^k_t)\},
\end{align}
is optimal in \eqref{eq:tildev}, whenever $\P_\pi(\widetilde\tau_{k+1}^*<\infty|Y_{\tau_k}=0)=1$.
Combining the latter fact with \eqref{eq:mu01} we get
\begin{align}\label{eq:mu01b}
\E_\pi\Big[1_{\{Y_{\tau_1}=0\}} V\big(\Pi^1_{\tau_{1}}
\big)\Big]&=\E_\pi\Big[1_{\{Y_{\tau_1}=0\}} \widetilde V^1 \big(\Pi^1_{\tau_{1}}
\big)\Big] \\
&\le\E_\pi\Big[1_{\{Y_{\tau_1}=0\}}\E_{\pi}\Big[\beta\int_{\tau_1}^{\tau_2}\Pi^1_t\ud t+(\cA V)(\Pi^1_{\tau_2})\Big|\cG^1_{\tau_1}\Big]\Big],
\end{align}
where $\tau_2$ is also taken from the same admissible sequence as $\tau_1$.
Combining \eqref{eq:mu01} and \eqref{eq:mu01b} we obtain
\begin{align}\label{eq:ub0}
\E_\pi\Big[(1 - \Pi^0_{\tau_1} + \epsilon \Pi^0_{\tau_{1}}) V\big(g(\Pi^0_{\tau_{1}}
)\big)\Big]\le \E_\pi\Big[1_{\{Y_{\tau_1}=0\}}\Big(\beta\int_{\tau_1}^{\tau_2}\Pi^1_t\ud t+(\cA V)(\Pi^1_{\tau_2})\Big)\Big].\notag
\end{align}
Plugging the final expression into \eqref{eq:v1} gives
\begin{align}
V(\pi) &\leq \E_\pi \Big[ 1 + \beta \int_{0}^{\tau_1} \Pi^0_t\ud t  + 1_{\{Y_{\tau_1}=0\}}\Big(\beta\int_{\tau_1}^{\tau_2}\Pi^1_t\ud t+(\cA V)(\Pi^1_{\tau_2})\Big) \Big].
\end{align}

Let us now argue by induction and assume that for some $n\in\N$
\begin{align}\label{eq:ubV}
\begin{aligned}
V(\pi)  \leq \E_\pi \Big[ &  \sum_{j=0}^{n-1} 1_{\{Y_{\tau_j}=0\}}\Big(1 + \beta \int_{\tau_{j}}^{\tau_{j+1}} \Pi^j_t\ud t\Big)  +1_{\{Y_{\tau_{n-1}}=0\}}\Big((\cA V)(\Pi^{n-1}_{\tau_{n}}) - 1\Big) \Big].
\end{aligned}
\end{align}
As in \eqref{eq:mu01} and \eqref{eq:mu01b}, using $g(\Pi^{n-1}_{\tau_n})=\Pi^{n}_{\tau_n}$ on $\{Y_{\tau_n}=0\}$ we obtain
\begin{equation}
\begin{aligned}
&1_{\{Y_{\tau_{n-1}}=0\}}\big(1-(1-\epsilon)\Pi^{n-1}_{\tau_n}\big)V\big(g(\Pi^{n-1}_{\tau_n})\big)\\
&=\E_\pi\Big[1_{\{Y_{\tau_n}=0\}}V\big(g(\Pi^{n-1}_{\tau_n})\big)\Big|\cG^{n-1}_{\tau_n}\Big]\\
&=\E_\pi\Big[1_{\{Y_{\tau_n}=0\}}\widetilde V^n\big(\Pi^{n}_{\tau_n}\big)\Big|\cG^{n-1}_{\tau_n}\Big]\\
&\le \E_\pi\Big[1_{\{Y_{\tau_n}=0\}} \E_{\pi}\Big[\beta\int_{\tau_n}^{\tau_{n+1}}\Pi^n_t\ud t+(\cA V)(\Pi^n_{\tau_{n+1}})\Big|\cG^{n}_{\tau_n}\Big]\Big|\cG^{n-1}_{\tau_n}\Big],
\end{aligned}
\end{equation}
where $\tau_{n+1}$ is taken from the same admissible sequence as $\tau_1\le \tau_2\le\ldots\le \tau_{n}$.
Substituting into \eqref{eq:ubV} we complete the induction step by tower property
\begin{equation}\label{eq:Vple}
\begin{aligned}
V(\pi) \leq \E_\pi \Big[ & \sum_{j=0}^{n} 1_{\{Y_{\tau_j}=0\}}\Big(1 + \beta \int_{\tau_{j}}^{\tau_{j+1}} \Pi^j_t\ud t\Big)  +1_{\{Y_{\tau_{n}}=0\}}\Big( (\cA V)(\Pi^{n}_{\tau_{n+1}}) - 1 \Big) \Big].
\end{aligned}
\end{equation}

Next we want to let $n\to\infty$. By monotone convergence
\begin{equation}\label{eq:limV}
\begin{aligned}
\lim_{n\to\infty}\E_\pi \Big[ \sum_{j=0}^{n} 1_{\{Y_{\tau_j}=0\}}\Big(1\! +\! \beta \int_{\tau_{j}}^{\tau_{j+1}}\! \Pi^j_t\ud t\Big)\Big]=\E_\pi \Big[ \sum_{j=0}^{\infty} 1_{\{Y_{\tau_j}=0\}}\Big(1\! +\! \beta \int_{\tau_{j}}^{\tau_{j+1}}\! \Pi^j_t\ud t\Big)\Big].
\end{aligned}
\end{equation}
For the remaining term, letting $\|\cdot\|_\infty$ be the $L^\infty(0,1)$-norm, the boundedness of $V$ allows us to write
\begin{equation}\label{eq:limV2}
\begin{aligned}
\lim_{n\to\infty}\E_\pi\Big[1_{\{Y_{\tau_{n}}=0\}}\Big( (\cA V)(\Pi^{n}_{\tau_{n+1}}) - 1\Big)\Big]\le  \|V\|_{\infty} \lim_{n\to\infty}\P_\pi\big(Y_{\tau_{n}}=0\big).
\end{aligned}
\end{equation}
For any $\delta>0$ there is $t_\delta>0$ such that $\P_\pi(\theta>t_\delta)<\delta$. Moreover, since any admissible sequence $(\tau_n)_{n=0}^\infty$ converges to infinity  in the sense of \eqref{eq:taun-lim}, there must also exist $n_\delta\in\N$ such that $\P_\pi(\tau_m\le t_\delta)<\delta$ for all $m\ge n_\delta$. Now, for any $n\ge m\ge n_\delta$
\begin{equation}
\begin{aligned}
\P_\pi\big(Y_{\tau_{n}}=0\big)&=\P_\pi\big(Y_{\tau_{n}}=0,\theta\le t_\delta\big)\!+\!\P_\pi\big(Y_{\tau_{n}}=0,\theta> t_\delta\big)\\
&\le \P_\pi\big(Y_{\tau_{n}}=0,\theta\le t_\delta,\tau_m\le t_\delta\big)
\!+\!\P_\pi\big(Y_{\tau_{n}}=0,\theta\le t_\delta,\tau_m> t_\delta\big)\!+\!\delta\notag\\
&\le \P_\pi\big(Y_{\tau_{n}}=0,\theta\le t_\delta,\tau_m> t_\delta\big)
\!+\!2\delta\notag\\
&= \P_\pi\big(Y_{\tau_{n}}=0,\theta\le t_\delta,\tau_m> t_\delta,\cap_{p=m}^n\{Z_p=0,\theta\le \tau_p\}\big)
\!+\!2\delta\notag\\
&\le \P_\pi\big(\cap_{p=m}^n\{U_p>1-\epsilon,\theta\le \tau_p\}\big)
\!+\!2\delta\le \epsilon^{n-m}+2\delta,\notag
\end{aligned}
\end{equation}
where in the penultimate inequality we use Remark~\ref{rem:intuition}-(c) and for the final one independence of $U_k$'s.
Then,
\[
\lim_{n\to\infty}\P_\pi(Y_{\tau_n}=0)\le 2\delta.
\]
By arbitrariness of $\delta$ and combining \eqref{eq:limV} and \eqref{eq:limV2} we arrive at
\[
V(\pi)\le \E_\pi \Big[ \sum_{j=0}^{\infty} 1_{\{Y_{\tau_j}=0\}}\Big(1 + \beta \int_{\tau_{j}}^{\tau_{j+1}} \Pi^j_t\ud t\Big)\Big].
\]
Since the admissible sequence $(\tau_j)_{j=0}^\infty$ was arbitrary we conclude $V(\pi)\le \hat V(\pi)$.

The remaining claim in the proposition is easily verified upon observing that for the sequence $(\tau^*_j)_{j=0}^\infty$ defined by \eqref{eq:taun} all the inequalities above become equalities. Indeed, $\bar\tau_{k+1}$ coincides with the optimal time $\widetilde\tau^*_{k+1}$ defined in \eqref{eq:tautilde}, and \eqref{eq:Vple} becomes
\begin{equation}\label{eq:Vple2}
V(\pi) = \E_\pi \Big[ \sum_{j=0}^{n} 1_{\{Y_{\tau^*_j}=0\}}\Big(1 + \beta \int_{\tau^*_{j}}^{\tau^*_{j+1}} \Pi^j_t\ud t\Big) +1_{\{Y_{\tau^*_{n}}=0\}}\Big\{(\cA V)(\Pi^{n}_{\tau^*_{n+1}}) - 1\Big\} \Big].
\end{equation}
Then, by letting $n\to\infty$ and using the same argument as above we conclude.
\end{proof}

We say that the solution to {\bf Problem [R]} is a {\em relevant solution} if $\tau^*_V$ defined in \eqref{eq:tauV}
is optimal and the sequence of stopping times defined by \eqref{eq:taun} is admissible. Then, by Proposition~\ref{prop:recur} we deduce that any {\em relevant solution} to {\bf Problem [R]} coincides with $\hat V$ and we have a simple corollary:
\begin{corollary}\label{cor:unique}
There can be at most one {\em relevant solution} $V$ to {\bf Problem [R]} and it must be $V=\hat V$.
\end{corollary}

\begin{remark}It may be worth noticing that, a priori, the solution $V$ of {\bf Problem [R]} is only bounded and measurable. This is not sufficient to deduce optimality of $\tau^*_V$ from standard optimal stopping theory (cf.\ \cite[Ch.\ 3, Sec.\ 3, Thm.\ 3]{shiryaev2007optimal}). We will need to check optimality of $\tau^*_V$ with a more detailed analysis in the next sections.

\end{remark}

\section{Explicit solution of the recursive problem}\label{sec:solution}

In this section we construct an explicit solution of {\bf Problem [R]}. To be precise we construct exactly the unique {\em relevant solution} mentioned in Corollary~\ref{cor:unique}. We proceed with two steps: first we formulate a verification result for a free boundary problem with a recursive boundary condition (Section~\ref{sec:verification}) and then we show that such free boundary problem admits a unique solution (Sections~\ref{sec:construct.sol.free.boundary} and \ref{sec:solve.recursive}).

\subsection{Free boundary problem: Formulation and verification theorem} \label{sec:verification}
The structure of the recursive problem on its own does not provide much intuition for the structure of the optimal stopping rule. However, the connection with the quickest detection problem in Proposition~\ref{prop:recur} suggests that it should be optimal to stop when there is a sufficiently high belief of the presence of a drift. That is, we should expect that $\tau^*_V$ be the first time $\bar \Pi$ exceeds a threshold $a_*\in(0,1)$.  That intuition motivates the form of the free boundary problem we consider next.

Given any $a\in (0,1)$,  recall the increasing sequence $\{g^{-k}(a),k\in\N\cup\{0\}\}$ with $g^0(a):=a$,  introduced after \eqref{eq:def.ginv}. Let us introduce the set
\begin{align}\label{eq:cDa}
\cD_a = \{ a, g^{-1}(a), g^{-2}(a), \dots\} \cap [0, 1).
\end{align}
Note that if $\epsilon = 0$, $\cD_a = \{a\}$, and if $\epsilon \in (0, 1)$, $\cD_a$ is countably infinite.
Given a set $J \subseteq \mathbb{R}$, we use $\cC^m(J)$ to denote the set of functions that are $m$-times continuously differentiable on $J$. The set of continuous functions on $J$ is simply denoted $\cC(J)$.
Recall $\cL$ defined in~\eqref{eq:inf-generator}, which is also the infinitesimal generator of the process $\bar \Pi$.

\begin{proposition}\label{prop:verif}
Assume there is a pair $(v,b)$, where $b\in(0,1)$ and $v$ is a concave function with
\[
v\in\cC\big([0,1]\big)\cap \cC^1\big((0,1)\big)\cap \cC^2\big((0,1)\setminus\cD_b\big)
\]
that satisfies
\begin{align}
&( \cL v  )(\pi) =   -\beta \pi,   \quad   \text{for all $\pi \in  (0, b)$},   \label{eq:c1} \\
&v(\pi) = (\cA v) (\pi),   \quad \text{for all  $\pi \in [b, 1)$},   \label{eq:c3} \\
& v'(b) =  ( \cA v)'(b), \label{eq:c4}\\
& \lim_{\pi \downarrow 0} | v'(\pi ) | < \infty,  \label{eq:c2}\\
&( \cL v  )(\pi) \ge   -\beta \pi,   \quad   \text{for  all $\pi \in (0,1) \setminus\cD_b$},   \label{eq:c1b} \\
& v(\pi)\le (\cA v)(\pi), \quad \text{for all $\pi\in[0,1)$}.\label{eq:obs}
\end{align}
Define $\tau_b=\inf\{t\ge 0 : \bar \Pi_t\ge b\}$.
Then $v$ is a solution of {\bf Problem [R]}
and the stopping time $\tau_b$
is $\bar \P$-a.s.\ finite and optimal.
\end{proposition}

\begin{remark}Notice that $\{0\}$ is an entrance boundary for $\bar \Pi$ \cite[Ch.\ 2]{borodin2015handbook}. Therefore we expect that a solution $V$ to {\bf Problem [R]} should satisfy the boundary condition $V'(0+)=0$.
But for the proof of Proposition~\ref{prop:verif}, we only need~\eqref{eq:c2}. Note that $\lim_{\pi \downarrow 0} v'(\pi)$
always exists because $v$ is concave by assumption.

\end{remark}

\begin{proof}
Consider {\bf Problem [R]} with $\pi = 1$ first. Then $\bar{\Pi}_t = 1$ for every $t \geq 0$ and the value function in~\eqref{eq:def.v} can be computed as
\begin{align*}
    V(1) = \inf_{\tau \in \cT(\bar\cF_t)} \bar{\E}_\pi \left[1 + \beta \tau + \epsilon V(1) \right] = 1 + \epsilon V(1)  + \inf_{\tau \in \cT(\bar\cF_t)} \bar{\E}_\pi [\beta \tau].
\end{align*}
Clearly, $\tau = 0$ is the optimal stopping time, from which we get $V(1) = 1/ (1 - \epsilon)$. So it remains to verify that $v(1) = 1/(1 - \epsilon)$.
By~\eqref{eq:c3} and the continuity of $v$,
\begin{align}\label{eq:v1.limit}
\begin{aligned}
   &  v(1) = \lim_{\pi \uparrow 1} v(\pi) = \lim_{\pi \uparrow 1} (\cA v)(\pi) \\
   &= 1 + \lim_{\pi \uparrow 1}   [1 - (1 - \epsilon) \pi] v (g (\pi) )  = 1 + \epsilon \lim_{\pi \uparrow 1} v(g  (\pi)).
\end{aligned}
\end{align}
If $\epsilon = 0$, then  $v(1) = 1 =V(1)$. If $\epsilon \in (0, 1]$, then $\lim_{\pi \uparrow 1} g(\pi) = 1$, and thus $v(1) = 1 + \epsilon v(1)$.
This proves that $v(1) = 1 / (1 - \epsilon) = V(1)$.

Now consider {\bf Problem [R]} with $\pi \in [0, 1)$.
By~\eqref{eq:c2} and the assumption $v \in \cC^1((0, 1))$, $v'$ is bounded on $(0, 1 - n^{-1}]$ for any $n\in\N \setminus \{1\}$. Thus, $v'$ is extended continuously to $[0,1)$ thanks to the concavity of $v$.
 We denote this extension by writing  $v\in \cC^1([0,1))$.
We further notice that for any $t\in(0,\infty)$ and $\pi\in[0,1)$, we have
$\bar\P_\pi(\bar\Pi_t\in(0,1))=1$. Finally, we set $\sigma_n = \inf\{s \geq 0: \bar{\Pi}_s \geq 1 - n^{-1}\}$, so that $\sigma_n\uparrow \infty$ as $n\to\infty$, $\bar\P_\pi$-a.s.

By applying It\^o-Tanaka's formula~\cite[Theorem 3.7.1]{karatzas2012brownian}  to $v( \bar \Pi_{t\wedge\sigma_n})$ and using   $v \in \cC^1([0, 1))$, we get
\begin{align}\label{eq:m1}
\begin{aligned}
v( \bar \Pi_{t\wedge\sigma_n})  =\;  v(\pi)   + \int_0^{t\wedge\sigma_n} ( \cL v  )(\bar \Pi_s)1_{\{\bar \Pi_s \not\in \cD_b\}} \ud s  + \frac{\mu}{\sigma}\int_0^{t\wedge\sigma_n} \bar \Pi_s (1 - \bar \Pi_s )v'(\bar \Pi_s) \ud \bar{W}_{s}.
\end{aligned}
\end{align}
Notice that in the first integral on the right-hand side above we used that $\int_0^t  1_{\{\bar \Pi_s \in \cD_b\}} \ud s = 0$ a.s.\ and that $\cD_b\cap[0,1-n^{-1}]$ is a finite set.
By~\eqref{eq:c1b}, $( \cL v  )(\pi) \geq -\beta \pi$ for $\pi \in (0, 1)\setminus \cD_b$.
It follows from~\eqref{eq:obs} and~\eqref{eq:m1} that, a.s.,
\begin{align}\label{eq:m3}
(\cA v)(\bar \Pi_{t\wedge\sigma_n}) \ge  v(\bar \Pi_{t\wedge\sigma_n}) \ge v(\pi) - \beta\int_0^{t\wedge\sigma_n}\bar \Pi_s \ud s + M_{t\wedge\sigma_n},
\end{align}
where
\[
M_t := \frac{\mu}{\sigma}\int_0^t \bar \Pi_s (1 - \bar \Pi_s )v'(\bar \Pi_s) \ud \bar{W}_{s}
\]
is a local martingale and a true martingale up to $\sigma_n$.

Let $\tau\in \cT(\bar\cF_t)$ be an arbitrary stopping time (recall that $\tau<\infty$ a.s.).  Hence,
$\bar{\E}_\pi[M_{\tau\wedge \sigma_n}] = 0$ for each $n \in \N$.
Replacing $t$ with $\tau$ in~\eqref{eq:m3} and taking $\bar \E_\pi$ on both sides, we obtain
 \begin{align}\label{eq:m4}
    \bar \E_\pi \left[ \beta\int_0^{\tau\wedge \sigma_n}\bar \Pi_s \ud s \right] + \bar \E_\pi[(\cA v)(\bar \Pi_{\tau\wedge \sigma_n}) ] \geq v(\pi).
\end{align}
Since $\bar{\Pi}\ge 0$, then $\lim_{n\to\infty}\bar \E_\pi[ \beta\int_0^{\tau\wedge \sigma_n}\bar \Pi_s \ud s] = \bar \E_\pi[ \beta\int_0^{\tau}\bar \Pi_s \ud s] $ by the monotone convergence theorem. Since $v$ is bounded and continuous on $[0, 1]$,  so is $\cA v$.
The dominated convergence theorem then implies that  \[\lim_{n\to\infty}\bar \E_\pi[(\cA v)(\bar \Pi_{\tau\wedge \sigma_n}) ] = \bar \E_\pi[(\cA v)(\bar \Pi_{\tau}) ].\]
Taking $n \rightarrow \infty$ in \eqref{eq:m4}, we get
 \begin{align}\label{eq:m5}
    \bar \E_\pi \left[ \beta\int_0^{\tau}\bar \Pi_s \ud s + (\cA v)(\bar \Pi_{\tau}) \right] \geq v(\pi),
\end{align}
for all $\tau\in \cT(\bar\cF_t)$.
Now consider $\tau_b$, which is finite a.s. because  $\lim_{t\to\infty}\bar \Pi_t=1$  a.s.
Take $n$ such that $b<1-n^{-1}$, so that $\tau_b\le \sigma_n$.
Then, using \eqref{eq:c1}, \eqref{eq:c3} and substituting $t$ with $\tau_b$ in \eqref{eq:m1} yields
\begin{align}\label{eq:m6}
 (\cA v)(\bar \Pi_{\tau_b}) = v( \bar \Pi_{\tau_b}) = v(\pi) - \beta\int_0^{\tau_b}\bar \Pi_s\ud s + M_{\tau_b}.
\end{align}
Taking expectation on both sides yields
\begin{align}\label{eq:m7}
v(\pi) = \bar{\E}_\pi \left[ \beta\int_0^{\tau_b}\bar \Pi_s \ud s + (\cA v)(\bar \Pi_{\tau_b}) \right].
\end{align}
It follows from~\eqref{eq:m5} and~\eqref{eq:m7} that $\tau_b$  is optimal and $v$ is equal to the value function given in~\eqref{eq:def.v}.
\end{proof}

\subsection{Solution of the free boundary problem} \label{sec:construct.sol.free.boundary}
The main goal of this section is to construct an explicit solution to the free boundary problem formulated in Proposition~\ref{prop:verif}. We denote it by $(u, a^*)$. We prepare the ground for this result, which is formally stated in Proposition~\ref{lm:recur}, by introducing relevant functions.

As shown in~\cite[Ch.\ 22.1.3 and Eq. (22.1.12)]{peskir2006optimal},  any solution to the ODE~\eqref{eq:c1} that satisfies condition~\eqref{eq:c2} can be expressed by
\begin{equation}\label{eq:Psi}
\Psi(\pi) +C = \int_0^\pi \psi(x) \ud x + C.
\end{equation}
where $C$ is some constant and $\psi$ was introduced in \eqref{eq:def.psi}.
Then it must hold that $u=\Psi + C$ on $[0, a^*)$. The pair $(C,a^*)$ can be determined by imposing \eqref{eq:c3} and \eqref{eq:c4} at the boundary point $\pi=a^*$.
That is,
\begin{align}\label{eq:sys}
\begin{aligned}
C + \Psi(a^*)  =\;& 1 + [1 - (1 - \epsilon) a^*] \left[ C +  \Psi(g(a^*)) \right],  \\
\psi(a^*) = \;& - (1 - \epsilon) \left[ C + \Psi(g(a^*)) \right] + \frac{\epsilon}{1 - (1 - \epsilon) a^*} \psi(g(a^*)).
\end{aligned}
\end{align}
Suppose that $(C, a^*)$ can be uniquely determined by solving~\eqref{eq:sys} and that $a^* \in (0, 1)$ (which we will show in Proposition~\ref{lm:recur}).
We then need to extend $u$ to $[a^*, 1)$ in a way consistent with \eqref{eq:c3}.
We split $[0, 1)$ into disjoint intervals  with breakpoints $\cD_{a^*}$ (see \eqref{eq:cDa}) by defining
\begin{equation}\label{eq:def.Ik}
J_0^* = [0, a^*) \text{ and } J_k^* = [g^{- (k-1)}(a^*),  g^{-k}(a^*) )  \text{ for } k \in \N,
\end{equation}
where we recall $g^0(a^*) = a^*$.
Observe that if $\pi \in J_k^*$ for $k \in \N$, then $g(\pi) \in J_{k - 1}^*$.
Since $u = \Psi + C$  on $J_0^*$,  the function $(\cA u)(\pi)$ is well-defined for $\pi\in  J_1^*$, and we set $u = \cA(\Psi + C)$ so that $u = \cA u$ on $J_1^*$.  Iterating this argument we obtain an extension of $u$ to $[a^*, 1)$ that satisfies \eqref{eq:c3}:
\begin{equation}\label{eq:expression.u}
u(\pi)=
\left\{
\begin{array}{ll}
\Psi(\pi)+C, & \text{ for } \pi\in[0,a^*),\smallskip\\
\big[\cA^k (\Psi + C)\big] (\pi), & \text{ for } \pi\in J_k^*,\, k\in\N,
\end{array}
\right.
\end{equation}
where we recall $\cA^k f=[\cA\circ\ldots\circ\cA]f$ is the $k$-fold application of the operator $\cA$ to a function $f$.
If $\epsilon = 0$,~\eqref{eq:expression.u} simplifies to $u(\pi)  = 1 + (1 - \pi) C$ for  $\pi \in [a^*, 1)$.  Since $\cup_{k \geq 0} J_k^* = [0, 1)$ by Remark~\ref{rem:g-1}, the expression in \eqref{eq:expression.u} defines $u$ on $[0, 1)$.

\begin{proposition}\label{lm:recur}
 The pair $(u,a^*)$ constructed in \eqref{eq:sys}--\eqref{eq:expression.u} is the unique solution pair for the free boundary problem \eqref{eq:c1}--\eqref{eq:c2} such that
\[
u\in  \cC([0,1])\cap\cC^1([0,1))\cap \cC^2([0,1)\setminus\cD)\ \
\text{and}\ \ a^* \in (0, 1),
\]
where $\cD =\cD_{a^*}$ as in \eqref{eq:cDa}.
\end{proposition}

Notice that conditions \eqref{eq:c1b} and \eqref{eq:obs} are not guaranteed by Proposition~\ref{lm:recur} and will have to be proven separately in Section~\ref{sec:solve.recursive}. The proof of the proposition requires three technical lemmas that provide useful bounds on functions $\psi$ and $\psi'$. We state them here and  postpone their proofs to the Appendix.

\begin{lemma}\label{lm:psi}
For any $\pi \in (0, 1)$, we have
\begin{align*}
-\frac{\pi}{1 - \pi} \left( 1 - \frac{\pi}{1 + \rho } \right) < \frac{\lambda}{\beta} \psi(\pi)  <  - \frac{\rho}{1 + \rho}\frac{\pi}{ 1 - \pi }.
\end{align*}
Consequently,
$\lim_{\pi \uparrow 1}   (1 - \pi) \psi(\pi) = -  \beta / ( \lambda + \gamma).$
\end{lemma}
\begin{proof}
See Section~\ref{sec:proof.lm.psi} in Appendix.
\end{proof}

\begin{lemma}\label{lm:u2}
For any $\pi \in (0, 1)$,
\begin{equation*}
  - \frac{ 1 }{ (1- \pi)^2}   < \frac{\lambda}{\beta}\psi'(\pi) < 0.
\end{equation*}
\end{lemma}
\begin{proof}
See Section~\ref{sec:proof.lm.u2} in Appendix.
\end{proof}

The next lemma, in particular, provides an estimate which will also be needed later on  to verify condition \eqref{eq:obs} in the free boundary problem.
\begin{lemma}\label{lm:key}
For any $\pi \in (0, 1)$ and $\epsilon \in [0, 1)$,
\[
(1 - \epsilon) \beta \pi + \lambda (1 - \pi) [ \psi(\pi) - \psi(g_\epsilon(\pi))] > 0 \text{  and  } \psi'(\pi) < \frac{\epsilon^2 \,  \psi'(g_\epsilon(\pi)) }{[1 - (1 - \epsilon )\pi]^3}.
\]
\end{lemma}
\begin{proof}
See Section~\ref{sec:proof.lm.key} in Appendix.
\end{proof}

We are now ready to prove Proposition~\ref{lm:recur}.

\begin{proof}[Proof of Proposition {\ref{lm:recur}}]
The function $u$ in \eqref{eq:expression.u} solves the system \eqref{eq:c1}--\eqref{eq:c2} by construction, provided that $(C, a^*)$ is a solution pair of \eqref{eq:sys}. Conversely, any solution pair $(\hat u,\hat a)$ of \eqref{eq:c1}--\eqref{eq:c2} is characterized by a pair $(\hat C,\hat a)$ that solves \eqref{eq:sys}. Then, $(u,a^*)$ is the unique solution pair of \eqref{eq:c1}--\eqref{eq:c2} if \eqref{eq:sys} admits a unique solution. It remains to verify that indeed $(C, a^*)$ is the unique solution of \eqref{eq:sys} and that $u\in \cC([0,1])\cap\cC^1([0,1))\cap \cC^2((0,1) \setminus\cD)$. Let us start by observing that continuity of $u$ at $\pi=1$ follows from $u=\cA u$ on $[a^*,1)$ and the same limit as in \eqref{eq:v1.limit}.

Expressing $C$ as a function of $a^*$, we find from~\eqref{eq:sys} that $a^*$ must be a solution to the equation $F(a) = 0$, where
\begin{equation}\label{eq:def.F}
F(a) =  1 + \Psi(g(a))  - g(a) \psi(g(a))   - \Psi(a)   + a \psi(a).
\end{equation}
Differentiating with respect to $a\in(0,1)$, we find that
\begin{align*}
F'(a) = \;&  a \psi'(a) - g(a) g'(a) \psi'(g(a)) =
a \left\{ \psi'(a) - \frac{\epsilon^2 \, \psi'(g(a)) }{[1 - (1 - \epsilon )a]^3}  \right\} < 0,
\end{align*}
where the final inequality is by  the second inequality in Lemma~\ref{lm:key}.  Since $g(0) = 0$ and $\psi(0 + )= 0$, we have $\lim_{a \downarrow 0}F( a )  = 1$.
Meanwhile, Lemma~\ref{lm:u2} implies that $\psi$ is monotone decreasing, so
\begin{align*}
F(a)  =\;&  1 - \int_{g(a)}^a \psi(x) dx + a \psi(a) - g(a) \psi(g(a))     \\
\leq \;& 1 - \psi(a) (a - g(a)) + a \psi(a) - g(a) \psi(g(a))   \\
= \;& 1 + g(a) [ \psi(a) - \psi(g(a)) ].
\end{align*}
Recalling that $g(1)=1$,
\begin{align*}
\lim_{a \uparrow 1} F(a) \leq \;& 1 + \lim_{a \uparrow 1} [ \psi(a) - \psi(g(a))].
\end{align*}
Since $g(a) / (1 - g(a)) = \epsilon a / (1 - a)$, we can apply Lemma~\ref{lm:psi} to get
\begin{equation*}
\begin{aligned}
\lim_{a \uparrow 1} \frac{1-a}{a}[ \psi(a) - \psi(g(a))]
= \lim_{a \uparrow 1} \frac{1-a}{a}\psi(a) - \epsilon \lim_{a \uparrow 1} \frac{1- g(a)}{g(a)} \psi(g(a)) =  - \frac{\beta}{\lambda + \gamma} (1 - \epsilon),
\end{aligned}
\end{equation*}
which shows that $ \psi(a) - \psi(g(a)) \rightarrow -\infty$ as $a \rightarrow 1$. Since we have shown that $F$ is strictly monotone decreasing with $\lim_{a \downarrow 0}F( a )  = 1$ and $\lim_{a \uparrow 1} F(a) = -\infty$,  we conclude that $F(a) = 0$ has a unique solution on $(0, 1)$, proving the existence and uniqueness of $a^*$.
Plugging $a^*$ back into the system \eqref{eq:sys} uniquely determines the constant $C$.

Next, we show by induction that $u\in \cC^1([0,1) )\cap \cC^2([0,1)\setminus\cD)$. We assume $\epsilon \in (0, 1)$; the argument for the case $\epsilon = 0$ is trivial because $u(\pi)=1+C(1-\pi)$ for $\pi\in[a^*,1]$.
Take $k \in \N$ and denote $J^\circ_k=(g^{-(k-1)}(a^*),g^{-k}(a^*))$. Set $J^\circ_0=(0,a^*)$. The inductive hypothesis reads
\def\induct{\textbf{(A)}}
\begin{flushleft}
\induct{} $u:[0,1)\to \R$ is such that  $u\in\cC^1\big([0,g^{-k}(a^*))\big) \cap \cC^2(\cup_{i=0}^{k+1}J^\circ_i)$.
\end{flushleft}
First we argue that \induct{} holds for $k = 0$.
Since $u$ restricted to $(0, a^*)$ is the solution to the ODE~\eqref{eq:c1}, then $u \in \cC^2((0, a^*))$. By Lemma~\ref{lm:u2}, $u''$ is bounded on $(0, a^*)$ and  it admits limits at both points $0$ and $a^*$. Thus $u$ has an extension in $\cC^2([0, a^*])$.
Further,~\eqref{eq:sys} guarantees that  $u$ is continuously differentiable at $a^*$. Since $u = \cA u $ on $J^\circ_1$, then $u\in\cC^2(J_0^\circ\cup J^\circ_1)$ by the regularity of $u$ on $J^\circ_0$. Hence, continuous differentiability at $a^*$ yields $u\in\cC^1 ([0,g^{-1}(a^*)) )$ as needed.

Now suppose \induct{} holds for an arbitrary $k \in \N \cup \{0\}$, and we show that \induct{} also holds for $k + 1$.
For any $\pi \in J_{k + 2}^*$ we have $g(\pi) \in J_{k+1}^*$ and, by~\eqref{eq:expression.u},
\begin{equation}\label{eq:stopIk}
    u(\pi) = 1 + [1 - (1 - \epsilon) \pi] u (g (\pi) ).
\end{equation}
Hence, $u \in \cC^2(J^\circ_{k+2})$ by the assumption~\induct{}. Now we need to show that $u$ is continuously differentiable at $g^{-k}(a^*)$.

It follows from~\eqref{eq:stopIk} that
\begin{equation}\label{eq:contu}
\begin{aligned}
\lim_{\pi \uparrow g^{-k}(a^*) }  u ( \pi ) = \;&  \lim_{\pi \uparrow g^{-k}(a^*) }\Big( 1 + [1 - (1 - \epsilon) \pi] u  (g (\pi) )  \Big)
= \; 1 +   [1 - (1 - \epsilon)  g^{-k}(a^*)  ]   \lim_{\pi \uparrow g^{-k}(a^*) } u  (g (\pi) ).
\end{aligned}
\end{equation}
By the induction hypothesis, $u$ is continuous at $g^{-(k-1)}(a^*)$, yielding the second equality below
\begin{align*}
 \lim_{\pi \uparrow g^{-k}(a^*) } u  (g (\pi) ) =  \lim_{\pi \uparrow g^{-(k-1)}(a^*) } u ( \pi )
 =  \lim_{\pi \downarrow g^{-(k-1)}(a^*) } u  ( \pi )   =  \lim_{\pi \downarrow g^{-k}(a^*) } u (g (\pi) ).
\end{align*}
Thus, combining with \eqref{eq:contu},
\begin{align*}
\lim_{\pi \uparrow g^{-k}(a^*) }  u( \pi ) =\;& 1 +   [1 - (1 - \epsilon)  g^{-k}(a^*)]\lim_{\pi \downarrow g^{-k}(a^*) } u (g (\pi) )\\
=\;&\lim_{\pi \downarrow g^{-k}(a^*) }\Big(1 +   [1 - (1 - \epsilon)  \pi] u (g (\pi) )\Big)=\lim_{\pi \downarrow g^{-k}(a^*) }  u( \pi ),
\end{align*}
where the final equality is by~\eqref{eq:stopIk} again. This shows continuity of $u$ across $g^{-k}(a^*)$, hence $u\in\cC([0,g^{-(k+1)}(a^*))$. Differentiating both sides of~\eqref{eq:stopIk} with respect to $\pi$, one can show by an analogous argument that the left and right derivatives of $u$ at $g^{-k}(a^*)$ coincide. Hence, assumption \induct{} holds for $k+1$.
By induction, we deduce that \induct{} holds for $k \in \N \cup \{0\}$;
that is, $u\in \cC^1([0,1) )\cap \cC^2([0,1)\setminus\cD)$.
\end{proof}

We next prove that the solution $u$ found in Proposition~\ref{lm:recur} is monotone decreasing and concave. These two results are needed to prove that \eqref{eq:c1b} holds (cf.\ proof of Theorem~\ref{thm:FBP}). Concavity is also needed for the verification result in Proposition~\ref{prop:verif}.
\begin{proposition}\label{prop:concave}
The pair $(u, a^*)$ considered in Proposition~\ref{lm:recur} satisfies
\begin{align}
\label{eq:b1} -\frac{\kappa }{1 - \pi}     <   u''(\pi) < 0,  \quad& \text{ for } \pi \in (0,1)\setminus \cD, \\
\label{eq:b2}\kappa \log (1 - \pi) < u'(\pi) < 0, \quad& \text{ for } \pi \in (0,1),
\end{align}
where $\kappa = \beta / ( \lambda (1 - a^*) )$ and $\cD = \cD_{a^*}$.
\end{proposition}
\begin{proof}
By \eqref{eq:c3} and \eqref{eq:expression.u}, $u(\pi)=(\cA u)(\pi)$ for any $\pi \in [a^*, 1)$. Differentiating both sides twice, thanks to Proposition~\ref{lm:recur} we obtain
\begin{equation} \label{eq:u2}
u''(\pi ) = (\cA u)''(\pi) =  \frac{\epsilon^2 \, u''(g(\pi))}{ [ 1 - (1 - \epsilon)\pi ]^3} ,\quad\text{for $\pi\in(a^*,1)\setminus\cD$}.
\end{equation}
From Lemma~\ref{lm:u2} we have $u''=\psi'<0$ on $(0, a^*)$. By~\eqref{eq:u2}, this implies that $u''(\pi) < 0$  for $\pi \in (a^*, g^{-1}(a^*))$.
By induction, $u''(\pi) < 0$  for every $\pi \in [0, 1) \setminus \cD$.
Further, using
\begin{align*}
\frac{1 - \pi}{ 1 - g(\pi)}    = 1 - (1 - \epsilon) \pi \geq \epsilon,
\end{align*}
we get from~\eqref{eq:u2} that
\begin{equation}
    |u''(\pi )| \leq \frac{1}{ 1 - (1 - \epsilon) \pi  }|u''(g(\pi))|  = \frac{1 - g(\pi)}{1 - \pi}|u''(g(\pi))|.
\end{equation}
Define $D(\pi) := |u''(\pi) (1 - \pi)|$. By the above inequality, $D(\pi) \leq D(g(\pi))$ for $\pi\in (a^*,1)\setminus\cD$.
For any $k \in \N$ and $\pi \in J_k^*$, we have $g^k(\pi) \in[g(a^*), a^*)$ and $D(\pi) \leq D(g(\pi)) \leq \cdots \leq D(g^k(\pi))$. This shows that
\[
\sup_{\pi \in (a^*,1)\setminus\cD } D(\pi) \le \sup_{\pi \in [g(a^*), a^*) } D(\pi) \leq \frac{\beta}{\lambda (1 - a^*) },
\]
thanks to Lemma~\ref{lm:u2}. Lemma~\ref{lm:u2} also gives
\[
\sup_{\pi \in [0, 1) \setminus\cD} D(\pi)=\max\Big(\sup_{\pi \in [0, a^*]} D(\pi), \sup_{\pi \in (a^*,1)\setminus\cD } D(\pi)\Big)\leq \frac{\beta}{\lambda (1 - a^*) }.
\]
Hence, \eqref{eq:b1} holds.

Since $\cD$ is countable and $u'(0) = 0$, then we can write $u'(\pi) = \int_0^\pi u''(x) \ud x$ for any $\pi \in (0, 1)$. The bounds in \eqref{eq:b2} now follow from \eqref{eq:b1}.
\end{proof}

\subsection{Solution of the recursive optimal stopping problem}\label{sec:solve.recursive}
We now state the main result of this section, which verifies that the pair $(u,a^*)$ constructed in the previous subsection  is the unique solution to the free boundary problem \eqref{eq:c1}--\eqref{eq:obs}.
\begin{theorem}\label{thm:FBP}
The pair $(u,a^*)$  constructed in \eqref{eq:sys}--\eqref{eq:expression.u} is the unique pair that solves the free boundary problem \eqref{eq:c1}--\eqref{eq:obs}. Moreover, $a^* \in (0, 1)$, $u$ is concave and
\[
u\in \cC([0,1])\cap\cC^1([0,1))\cap \cC^2([0,1)\setminus\cD),
\]
where $\cD =\cD_{a^*}$ as in \eqref{eq:cDa}.
\end{theorem}

\begin{proof}
Let $(u, a^*)$ be as given in Proposition~\ref{lm:recur}. By Proposition~\ref{lm:recur}, $a^* \in (0, 1)$ and $u \in \cC([0,1])\cap \cC^1([0,1))\cap \cC^2([0,1)\setminus\cD)$ are uniquely determined.
Since $\cD$ is countable, then $u$ is concave by Proposition~\ref{prop:concave}.
It only remains to show that $(u, a^*)$ also  satisfies~\eqref{eq:c1b} and~\eqref{eq:obs}.

Consider \eqref{eq:c1b} first. By construction the equality in \eqref{eq:c1b} holds on $J_0^* = [0, a^*)$, and we need to show that $(\cL u)(\pi +) \geq - \beta \pi$ for $\pi \in [a^*, 1) = \cup_{k \geq 1} J_k^*$.
To simplify the notation, let
\[H(\pi) = (\cL u)(\pi +) + \beta \pi
= \lambda(1-\pi) u'(\pi)+\gamma\pi^2(1-\pi)^2\thinspace u''(\pi +) + \beta \pi.\]
Notice that $H$ is well-defined for all $\pi\in(0,1)$ since $u\in \cC^1([0,1))\cap \cC^2([0,1)\setminus\cD)$  and $u''(\pi+)$ is defined at all points via \eqref{eq:u2}. For any $\pi \in [a^*, 1)$, we have
\begin{align*}
H( g(\pi) ) =&  \lambda\big(1 - g(\pi)\big)  u'(g(\pi))  +  \gamma g(\pi)^2(1-g(\pi))^2\thinspace u''(g(\pi) +) + \beta g(\pi),
\end{align*}
so that, by the definition of $g$
\begin{align*}
[1 - (1 - \epsilon) \pi]  H( g(\pi) ) =  \lambda(1 - \pi)  u'(g(\pi))  + \frac{\gamma \epsilon^2 \pi^2 (1 - \pi)^2  u''(g(\pi) +)  }{  [1 - (1 - \epsilon) \pi]^3} + \beta \epsilon  \pi.
\end{align*}
It then follows from~\eqref{eq:u2} that
\begin{align*}
[1 - (1 - \epsilon) \pi]  H( g(\pi) ) =  \;&
 \lambda(1 - \pi)  u'(g(\pi)) +   \gamma  \pi^2 (1 - \pi)^2   u''(\pi+) + \beta \epsilon \pi \\
 =\;& H(\pi)  +  \lambda(1 - \pi)  [ u'(g(\pi)) - u'(\pi) ] - (1 - \epsilon) \beta \pi.
\end{align*}
Dividing both sides   by $1 - \pi$ and using  $1 - g(\pi) = (1 - \pi)/ [ 1 - (1 - \epsilon) \pi]$,  we get
\begin{align}\label{eq:H1}
\frac{  H( g(\pi) ) }{ 1 - g(\pi)} =  \frac{ H(\pi) }{1 - \pi}  -  R(\pi),
\end{align}
where
\begin{align*}
  R(\pi) :=  \lambda  [  u'(\pi)- u'(g(\pi))] + \frac{(1 - \epsilon) \beta \pi}{1 - \pi}.
\end{align*}
Since $u \in \cC^1( [0, 1) )$,   $R(\pi)$ is also continuous on $[0, 1)$. Further,  for any $\pi \in [a^*, 1) \setminus \cD$,  by~\eqref{eq:u2},
\begin{equation*}
\begin{aligned}
\frac{\partial [  u'( \pi ) - u'(g(\pi))  ] }{\partial \pi} &=   u''(\pi) -
\frac{ \epsilon }{ [ 1 - (1 - \epsilon) \pi ]^2} u''( g(\pi) ) \\
&=  - \frac{ (1 - \pi) (1 - \epsilon)}{\epsilon} u''(\pi) > 0,
\end{aligned}
\end{equation*}
where the final inequality is due to Proposition~\ref{prop:concave}.
As $\pi/(1 - \pi)$ is clearly increasing in $\pi$,   $R$ is continuous and monotone increasing on $[a^*, 1)$, although $R$ is not differentiable at points in $\cD$.
By the first equation in Lemma~\ref{lm:key} and recalling that $u'=\psi$ in $[0,a^*]$,  we have $R(\pi)>0$ for any $\pi \in [0, a^*]$. By monotonicity of $R$ on $[a^*, 1)$, we conclude that $R(\pi) > 0$ for every $\pi \in [0, 1)$.
Since~\eqref{eq:c1} implies that $H(\pi) = 0$ for $\pi \in J_0^*$,  the left-hand side of \eqref{eq:H1} must be zero for $\pi\in J_1^*$. Thus, from the right-hand side of~\eqref{eq:H1} we have $H(\pi)= (1 - \pi) R(\pi) > 0$ for every $\pi \in J_1^*$. Arguing by induction, assume $H(\pi)>0$ for $\pi\in J_k^*$. Then, taking $\pi\in J_{k+1}^*$ in \eqref{eq:H1}, the left-hand side of the equation must be strictly positive. Thus, using $R > 0$, we find that $H(\pi) > 0$ for $\pi\in J_{k+1}^*$. Hence, $H(\pi) > 0$ for every $\pi \in [a^*, 1)$.

Finally we prove that  $u$ satisfies \eqref{eq:obs}. The inequality is  an equality on $[a^*,1]$ by \eqref{eq:expression.u}, so we only need to prove~\eqref{eq:obs} on $J_0^*$.
Differentiating twice the expression for $\cA u$ gives
\begin{align}\label{eq:Au}
(\cA u)''(\pi) =  \frac{\epsilon^2 }{ [ 1 - (1 - \epsilon)\pi ]^3} u''(g(\pi)  ),
\end{align}
for any $\pi$ such that $u''(g(\pi))$ exists.
We claim that   $(\cA u)''   \geq u''$ on $J_0^*$.
Recall that $u' = \psi$ on $J_0^*$ and $g(\pi) \leq  \pi$, which implies  $u''(g(\pi)) = \psi'(g(\pi))$  for any $\pi \in J_0^*$.  By the second equation in Lemma~\ref{lm:key}, $(\cA u)''\geq u'' $  on $J_0^*$.
Since $(\cA u)'(a^*)=u'(a^*)$ by \eqref{eq:c4}, then $(\cA u)'(\pi)\le u'(\pi)$ for $\pi\in[0,a^*]$. The latter inequality, combined with the boundary condition $(\cA u)(a^*)=u(a^*)$, implies $u(\pi)-(\cA u)(\pi)\le 0$ for $\pi\in [0,a^*)$, as needed.
\end{proof}

\section{Solution of the quickest detection problem and further remarks}\label{sec:solQDP}
Building upon the results of the previous sections we finally obtain the solution of the quickest detection problem. First, we state a uniqueness result that connects the pair $(u,a^*)$ to the solution of the quickest detection problem (recall the concept of {\em relevant solution} stated before Corollary~\ref{cor:unique}). Its proof uses Theorem~\ref{thm:FBP}, Proposition~\ref{prop:verif} and the explicit construction of the admissible sequence $(\tau^*_n)_{n\in\N}$ in \eqref{eq:taun}.

\begin{proposition}\label{prop:optimal}
The pair $(u,a^*)$  constructed in \eqref{eq:sys}--\eqref{eq:expression.u}
is the unique {\em relevant solution} to {\bf Problem [R]}. Thus $u=\hat V$, where $\hat{V}$ is given in~\eqref{eq:Vhat}.
\end{proposition}

\begin{proof}
By Theorem~\ref{thm:FBP}, $(u,a^*)$ solves uniquely the free boundary problem \eqref{eq:c1}--\eqref{eq:obs}. Then $u=V$ is a solution to {\bf Problem [R]} and
\[\tau^*_V=\tau^*_u=\tau_{a^*}=\inf\{t\ge 0: \bar \Pi_t\ge a^*\}\] is optimal by Proposition~\ref{prop:verif}.
If we prove that the sequence $(\tau^*_n)_{n\in\N}$ defined in \eqref{eq:taun} is admissible, then uniqueness will follow from Proposition~\ref{prop:recur} and Corollary~\ref{cor:unique}.

Measurability and monotonicity of $(\tau^*_n)_{n\in\N}$ are obvious. It is also clear that $\P_\pi(\tau_n^*<\infty)=1$ because $\lim_{t\to\infty}\Pi^n_t=1$, $\P_\pi$-a.s., for all $n\in\N$. It only remains to check condition \eqref{eq:taun-lim}.
Let $\bar N_Y=N_Y((\bar \tau_n))$ as in \eqref{eq:KY} with $\bar \tau_n$ from \eqref{eq:bartau},
i.e., $\bar N_Y$ is the number of tests along the sequence $(\bar \tau_n)_{n\in\N}$ until the drift is detected. It is then clear that by construction $\lim_{n\to\infty}\tau^*_n=\infty$ on the event $\{\bar N_Y<\infty\}$. We now show that $\P_\pi(\bar N_Y=\infty)=0$, which will conclude the proof.

Consider $\P_\pi(\bar N_Y \geq  k)$ for some arbitrary $k \in \N$.
Since \[\{\bar N_Y \geq k\} = \{ Z_1 = \cdots = Z_k = 0\},\]
we have
\begin{align*}
    \P_\pi(\bar N_Y \geq  k) =   \prod_{j = 1}^k  \P_\pi( Z_j = 0 | Z_1 = \cdots = Z_{j-1} = 0) = \prod_{j = 1}^k  \P_\pi( Z_j = 0 | Y_{\bar \tau_{j-1}} = 0).
\end{align*}
By the tower property  and \eqref{eq:YkGk},  we get
\begin{equation}
\begin{aligned}
    \P_\pi( Z_j = 0 | Y_{\bar \tau_{j-1}} = 0)
    = \;& \frac{\E_\pi \big[  \P_\pi \big(Z_j = 0 | \cG^{j-1}_{\bar \tau_j} \big) 1_{\{ Y_{\bar \tau_{j-1}} = 0 \}} \big] }{\P_\pi(Y_{\bar \tau_{j-1}} = 0)}
    =   \E_\pi \big[ 1 - (1 - \epsilon) \Pi^{j-1}_{\bar \tau_j}  | Y_{\bar \tau_{j-1}} = 0  \big].
\end{aligned}
\end{equation}
If $\pi \leq a^*$, then the definition of  $\bar{\tau}_j$ implies that $\Pi^{j-1}_{\bar \tau_j}=a^*$, given $Y_{\bar \tau_{j-1}} = 0$. If $\pi > a^*$, on the event $\{Y_{\bar \tau_{j-1}} = 0\}$ there are two possibilities:
\begin{itemize}
    \item $\Pi^{j-1}_{\bar \tau_{j-1}}\le a^*$ and therefore $\Pi^{j-1}_{\bar \tau_{j}}=a^*$, or
    \item $\Pi^{j-1}_{\bar \tau_{j-1}}> a^*$ and therefore $\bar\tau_j=\bar\tau_{j-1}$ so that $\Pi^{j-1}_{\bar \tau_{j}}=\Pi^{j-1}_{\bar \tau_{j-1}}>a^*$ (in this case $\Pi^{j}_{\bar \tau_{j}}=g(\Pi^{j-1}_{\bar \tau_{j}})<\Pi^{j-1}_{\bar \tau_{j}}$ if also $Z_j=0$; intuitively, tests are performed at time $\bar \tau_{j-1}$ until either a positive test outcome occurs or the posterior probability drops below $a^*$).
\end{itemize}
Hence, in all cases $\P_\pi( Z_j = 0 | Y_{\bar \tau_{j-1}} = 0)  \leq  1 - (1 - \epsilon)a^*$ for each $j$. In particular, the equality holds for each $j$ if $\pi \leq a^*$.
It follows that \[\P_\pi(\bar N_Y \geq  k) \leq ( 1- (1 - \epsilon)a^* )^k,\] which goes to zero as $k$ tends to infinity. Indeed, this shows that if  $\pi \leq a^*$ then  $\bar N_Y$ follows a geometric distribution with success probability $(1 - \epsilon) a^*$ (cf.\ Remark~\ref{rmk:hitting-Psi}).
\end{proof}

We are now ready to provide a formal proof of Theorem~\ref{thm:main}. For completeness, and for the ease of exposition, we formulate a slightly extended version of Theorem~\ref{thm:main} that also accounts for the connections between the value function of the quickest detection problem and the solution of the free boundary problem.
Recall the sets $J_0^*, J_1^*, \dots$ defined in~\eqref{eq:def.Ik}.
\begin{theorem}\label{cor:V}
Let $(u,a^*)$ be the unique pair, constructed in \eqref{eq:sys}--\eqref{eq:expression.u},
that solves the free boundary problem \eqref{eq:c1}--\eqref{eq:obs}. Then, the value function $\hat V$ of the quickest detection problem is equal to $u$ and   given by
\begin{align}
\hat V(\pi)=
\left\{
\begin{array}{ll}
\Psi(\pi)+C, & \pi\in[0,a^*),\smallskip\\
\big[\cA^k (\Psi + C)\big] (\pi), & \pi\in J_k^*,\, k\in\N,
\end{array}
\right.
\end{align}
with $\Psi$ defined in \eqref{eq:Psi} and the constant
\begin{equation}
    C = \frac{1}{(1-\epsilon) a^*} + \frac{1 - (1 - \epsilon)a^* }{ (1 - \epsilon)a^* }\Psi( g(a^*) )  - \frac{1}{(1 - \epsilon) a^*} \Psi(a^*).
\end{equation}
The boundary $a^*$ is the unique solution in $(0,1)$ of
\begin{align} \label{eq:a*eq}
 1 + \Psi(g(a))  - g(a) \psi(g(a))   - \Psi(a)   + a \psi(a) = 0,
\end{align}
and the sequence $(\tau^*_n)_{n=0}^\infty$ given by: $\tau^*_0=0$,
\[
\tau^*_{n} =\bar\tau_n 1_{\{Y_{\tau^*_{n-1}}=0\}}+ (\tau^*_{n-1}+1) 1_{\{Y_{\tau^*_{n-1}}>0\}},
\]
with
\begin{align}
\bar \tau_{n} = \inf\{t \geq  \tau^*_{n-1}:  \Pi^{n-1}_{t}\ge a^*\},
\end{align}
for $n\in\N$, is admissible and optimal for $\hat V$.
\end{theorem}
\begin{proof}
The fact that $u=\hat V$ and the optimality of the sequence of stopping times follow from Proposition~\ref{prop:optimal} (see also Proposition~\ref{prop:recur}).
The expression for $C$ follows by rearranging the two equations in the system~\eqref{eq:sys}.
Finally, the fact that $a^*$ uniquely solves \eqref{eq:a*eq} is shown in the  proof of Proposition~\ref{lm:recur}.
\end{proof}

For a numerical example, we plot the value function $\hat{V} $ with $\lambda = 2$, $\gamma = 0.5$ and $\beta = 1$ in Fig.~\ref{fig:value}.
In line with Proposition~\ref{prop:concave}, for each choice of $\epsilon$, $V \colon [0, 1] \rightarrow \R$ is a concave and monotone decreasing function.
As shown in the proof of Proposition~\ref{prop:verif}, letting $\pi \uparrow 1$ the value converges to $1/(1 - \epsilon)$.
We also observe that for each fixed $\pi$, the value increases with $\epsilon$. That tallies with the intuition that the larger $\epsilon$ the longer it takes to detect the drift.
\begin{figure}
\centerline{\includegraphics[width=0.5\linewidth]{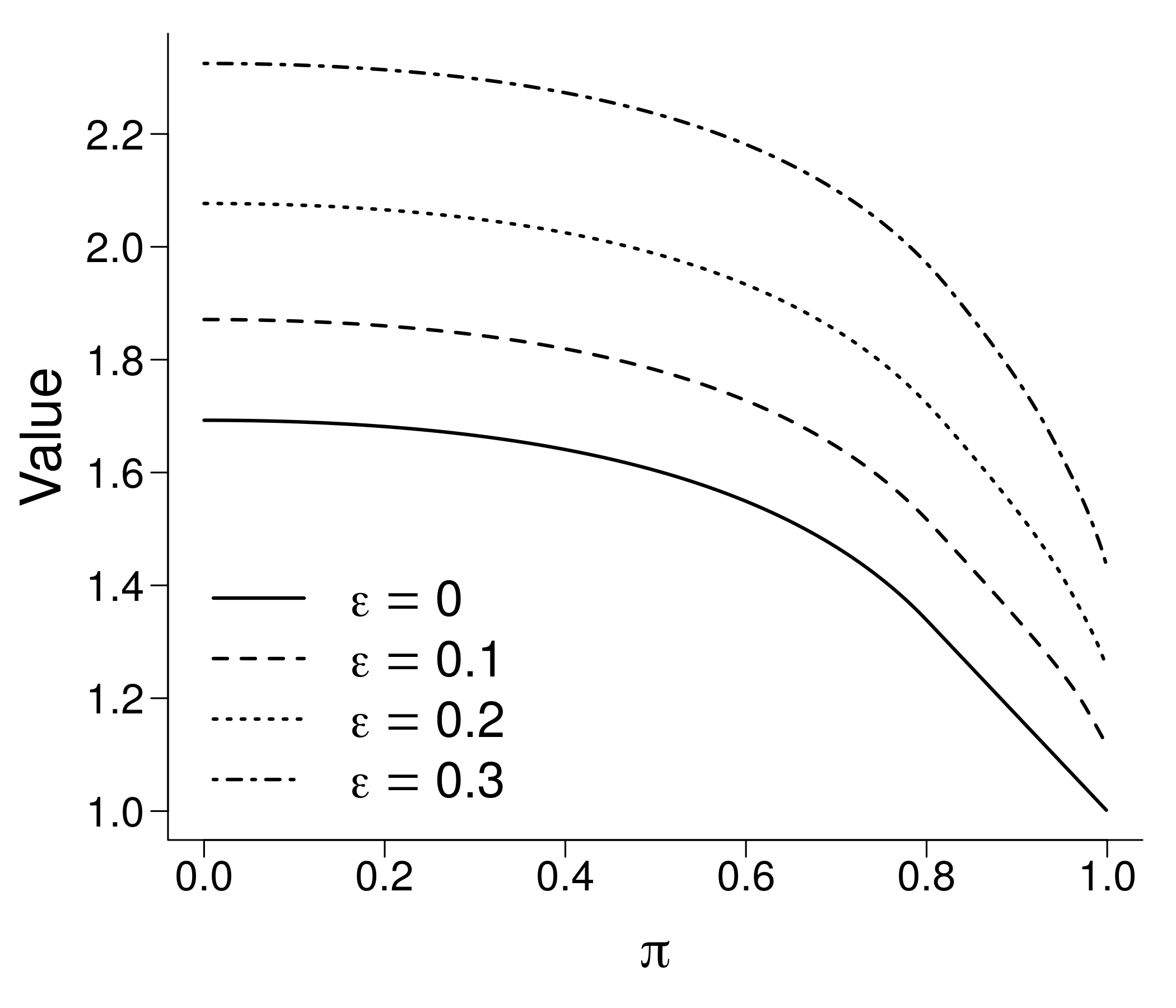}}
\caption{Value function $\hat{V}$ with $\lambda = 2$, $\gamma = 0.5$ and $\beta = 1$.} \label{fig:value}
\end{figure}

\subsection{Properties of the optimal stopping rule}

We now fix parameters $\lambda, \beta, \gamma$ and present  a series of results that characterize how $\epsilon$ affects the optimal stopping rule.
First, in Lemma~\ref{lm:n1}, we show  that the optimal stopping boundary $a^*$ increases with $\epsilon$.
We write $a^* = a^*(\epsilon)$ to indicate that $a^*$ is treated as a function of $\epsilon$.

\begin{lemma}\label{lm:n1}
The mapping $\epsilon\mapsto a^*(\epsilon)$ is strictly increasing for $\epsilon \in [0, 1)$,
and $\lim_{\epsilon \rightarrow 1} a^*(\epsilon) = 1$.
\end{lemma}
\begin{proof}
Recall that $a^*$ is determined by solving $F(a) = 0$ with $F$ given in~\eqref{eq:def.F}. Now we write $F_\epsilon$ to emphasize its dependence on $\epsilon$; that is,
\begin{align*}
F_{\epsilon}(a) = 1 + \Psi(g_{\epsilon}(a)) - g_{\epsilon}(a)\psi(g_{\epsilon}(a)) - \Psi(a) + a\psi(a) .
\end{align*}
Differentiating $\Psi(a) - a\psi(a)$ with respect to $a$ gives $- a \psi'(a) > 0$, for any $a \in (0, 1)$, by Lemma~\ref{lm:u2}. Hence, $a\mapsto \Psi(a) - a\psi(a)$ is increasing.

Fix an arbitrary $a$ in $(0,1)$ and let $0\leq\epsilon_{1} < \epsilon_{2} < 1$. Then $g_{\epsilon_{1}}(a) < g_{\epsilon_{2}}(a)$ and it follows that  $F_{\epsilon_{1}}(a) < F_{\epsilon_{2}}(a)$ for any $a\in (0,1)$. Since we have shown in the proof of Proposition~\ref{lm:recur} that $a\mapsto F_{\epsilon}(a)$ is continuous and decreasing on $(0,1)$,   we conclude that $a^*(\epsilon_1) < a^*(\epsilon_2)$; that is, $a^*$ is strictly increasing with $\epsilon$ for $\epsilon \in [0, 1)$.

To prove the second part of the lemma,
first note that $\lim_{\epsilon \rightarrow 1} a^*(\epsilon)$ exists and must be in $(0, 1]$ since $a^*(\epsilon) \in (0, 1)$ by Proposition~\ref{lm:recur}.
But for  any fixed $a \in (0, 1)$,  we have  $\lim_{\epsilon \rightarrow 1} F_\epsilon(a) = 1$  since $\lim_{\epsilon \rightarrow 1} g_\epsilon(a) = a$, which implies that the limit of $a^*(\epsilon)$ has to be $1$ as $\epsilon$ tends to $1$.
\end{proof}

\begin{remark}\label{rmk:epsilon-to-one}
The behavior of $a^*(\epsilon)$ for $\epsilon$ near $1$ could be undesirable for practical purposes (though such cases are uncommon in most applications).
In reality, regardless of how large $\epsilon$ is,  some action should be taken once the posterior probability of the presence of  disorder becomes sufficiently large. To address this issue, we will discuss in Section~\ref{sec:disc} an extension of our problem that could ensure a more intuitive decision rule even when $\epsilon$ is large.
\end{remark}

Next, we study the average drift-detection time under the optimal strategy. In particular, we look at how $\E_\pi [ \tau^*_{N_Y} ]$ depends on $\epsilon$. We observed
in Remark~\ref{rmk:hitting-Psi} that, for any initial value $\pi \leq a^*$, $N_Y$ is geometrically distributed with success probability $(1 - \epsilon) a^*$. Thus $\E_\pi[ N_Y] = ((1 - \epsilon) a^*)^{-1} $ for $\pi \in [0, a^*]$.
With the notation from Theorem~\ref{cor:V}, setting
\[\tau^*_{N_Y}=\inf\{\tau_n^*:Y_{\tau^*_n}=1\} \ \text{ and }\ \bar \tau_{N_Y}=\inf\{\bar \tau_n:Y_{\bar \tau_n}=1\},\]
we have
\begin{align*}
\E_\pi[\tau^*_{N_Y}]&=\E_\pi[\bar \tau_{N_Y}]=\E_\pi\Big[\sum_{k=0}^{N_Y-1}(\bar\tau_{k+1}-\bar\tau_k)\Big] \\ &=\E_\pi[\bar \tau_1]+\E_\pi\Big[\sum_{k=1}^{\infty}1_{\{N_Y\ge k+1\}}(\bar\tau_{k+1}-\bar\tau_k)\Big]\\
&=\E_\pi[\bar \tau_1]+\E_\pi\Big[\sum_{k=1}^{\infty}1_{\{N_Y\ge k+1\}}\E_\pi\big[\bar\tau_{k+1}-\bar\tau_k|\cG^{k}_{\bar \tau_k}\big]\Big],
\end{align*}
where we used that $\{N_Y\ge k+1\}\in\cG^{k}_{\bar \tau_k}$ and the monotone convergence theorem in the last step.  Using the explicit dynamics given by~\eqref{eq:Pik} and \eqref{eq:Pibar} and recalling $\tau_{a^*} = \inf\{ t\ge 0 : \bar{\Pi}_t = a^* \}$, we conclude that
$\E_\pi[\bar \tau_1]=\bar \E_\pi[ \tau_{a^*}]$ and
\[
1_{\{N_Y\ge k+1\}}\E_\pi\big[\bar\tau_{k+1}-\bar\tau_k|\cG^{k}_{\bar \tau_k}\big]=1_{\{N_Y\ge k+1\}}\bar \E_{g(a^*)}\big[\tau_{a^*}\big].
\]
Combining the above formulae we obtain
\begin{equation}\label{eq:wald}
\begin{aligned}
    \E_\pi [ \tau^*_{N_Y} ] =\;\bar \E_\pi[ \tau_{a^*} ]+(\E_\pi[N_Y]-1)\bar \E_{g(a^*)}[\tau_{a^*}]=\;\bar \E_\pi[ \tau_{a^*} ]+\left( \frac{1}{(1 - \epsilon)a^*} - 1 \right)\bar \E_{g(a^*)}[\tau_{a^*}].
\end{aligned}
\end{equation}

Setting
\begin{align*}
\chi(\pi) = -\frac{1}{\gamma}e^{-h(\pi)} \int_{0}^{\pi} \frac{e^{h(x)}}{x^2{(1-x)^2}} \ud x ,
\end{align*}
Lemma~\ref{lm:g3} gives the explicit expression for $\bar \E_\pi[ \tau_{a^*}]$.
\begin{lemma}\label{lm:g3}
It holds
\[
\bar \E_\pi[  \tau_{a^*} ] = -\int^{a^*}_{\pi} \chi(x) \ud x.
\]
\end{lemma}
\begin{proof}
Set $G(\pi) = \bar\E_\pi[\tau_{a^*}]$ for $\pi \in [0, a^*]$. It is well known that $G$ satisfies $\cL G = - 1$, with $G'(0+) < \infty$ and $G(a^*) = 0$.
Solving $\cL G = - 1$ and imposing $G'(0+) < \infty$, we find that $G' = \chi$. From the condition $G(a^*) = 0$, we get $G(\pi) =- \int^{a^*}_\pi \chi(x) \ud x$.
\end{proof}

We numerically investigate how $(1 - \epsilon)a^*(\epsilon)$,  $\bar \E_\pi[ \tau_{a^*}]$ and $\E_\pi [ \tau_{N_Y} ]$ change with $\epsilon$. We fix $\lambda = 2$ and $\gamma = 0.5$ and let $\beta = 0.1, 1, 10, 100$.
In Fig.~\ref{fig:astar}, we study how $a^*(\epsilon)$, $g_\epsilon(a^*(\epsilon))$ and $a^*(\epsilon) - g_{\epsilon}(a^*(\epsilon))$ change  as  functions of $\epsilon$.
As expected from  Lemma~\ref{lm:n1}, $a^*$ increases monotonically with $\epsilon$,  which means that as $\epsilon$ increases it is optimal for the observer to wait longer before performing the first test.
This further implies that $\epsilon \mapsto g_\epsilon(a^*(\epsilon))$ is also increasing, since $g_\epsilon(\pi)$ is monotone increasing in both $\epsilon$ and $\pi$.
However, $a^* - g(a^*)$ decreases with $\epsilon$. Thus, under the optimal stopping strategy, upon observing a negative outcome of a test, the change in the belief of the observer is small when $\epsilon$ is large.

\begin{figure}
\centerline{\includegraphics[width=0.99\linewidth]{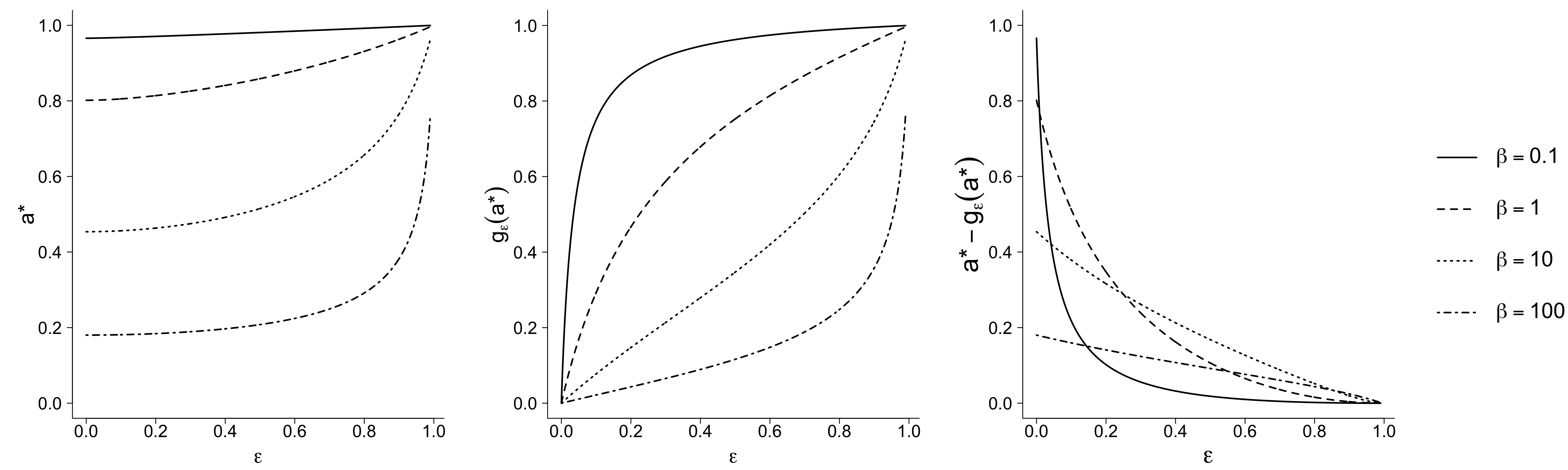}}
\caption{Optimal stopping rule $a^*$ for $\lambda = 2$, $\gamma = 0.5$ and $\epsilon \in [0, 0.99]$.  The line type indicates the choice of $\beta$.
    Left: $\epsilon \mapsto a^*(\epsilon)$; middle: $\epsilon \mapsto g_\epsilon(a^*(\epsilon))$; right:   $\epsilon \mapsto    a^*(\epsilon) - g_{\epsilon}(a^*(\epsilon))$.
    } \label{fig:astar}
\end{figure}

Fig.~\ref{fig:wait} visualizes the properties of the optimal stopping times under the same parameter setting as that of Fig.~\ref{fig:astar}.
First, the left panel shows that $\E_0[N_Y] = ((1 - \epsilon)a^*(\epsilon))^{-1}$ (i.e., the average number of tests needed to detect the drift-change) increases with $\epsilon$; that is, $\epsilon \mapsto ((1 - \epsilon)a^*(\epsilon))$ is decreasing though $\epsilon \mapsto a^*(\epsilon)$ is increasing.
In the middle panel, we plot the expected waiting time between first and second tests, conditional upon a negative outcome of the first test, as a function of $\epsilon$ (i.e., $\bar \E_{g(a^*)}[\tau_{a^*}]$, calculated by using Lemma~\ref{lm:g3}). Since in the right panel of Fig.~\ref{fig:astar} we have already seen that $a^* - g(a^*)$ decreases with $\epsilon$, then it is not surprising to observe in Fig.~\ref{fig:wait} that the expected waiting time between two consecutive tests also decreases when $\epsilon$ increases. That is,  the observer with lower accuracy will test the system more often than the observer with higher accuracy.
In the right panel of Fig.~\ref{fig:wait}, we plot $\E_\pi[\tau^*_{N_Y}]$ with $\pi = 0$, which is the expected time needed to detect the change in drift and can be calculated by~\eqref{eq:wald}. We see that $\E_0[\tau^*_{N_Y}]$ increases with $\epsilon$, suggesting that the observer with higher accuracy can detect and confirm the change faster than the observer with lower accuracy.
We have also tried other values for $\pi$ and made the same observation.

\begin{figure}
\centerline{\includegraphics[width=0.99\linewidth]{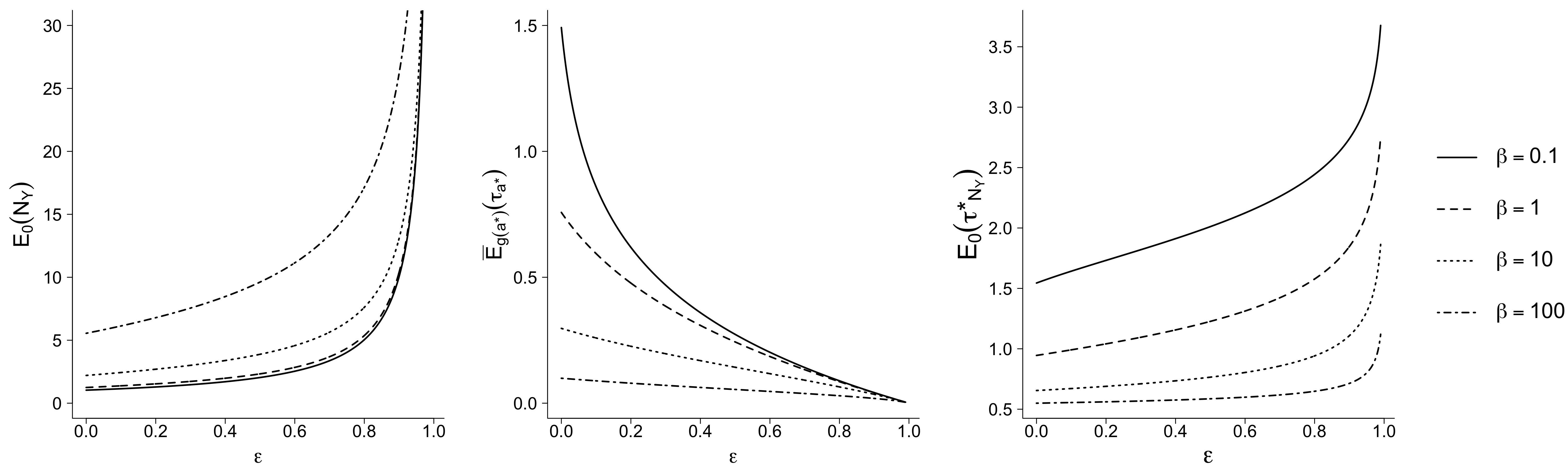}}
\caption{Optimal stopping times for $\lambda = 2$, $\gamma = 0.5$ and $\epsilon \in [0, 0.99]$.  The line type indicates the choice of $\beta$.
    Left: $\E_0[N_Y] = ( (1-\epsilon) a^* )^{-1}$, the expected number of tests needed to detect the drift-change;
    middle:  $\bar{\E}_{g(a^*)}[\tau_{a^*}]$,
    the expected waiting time between consecutive tests; right: $\E_0[\tau^*_{N_Y}]$, the expected time of the drift-change detection, with $\pi = 0$.
    } \label{fig:wait}
\end{figure}

\subsection{Analysis of the special case with \texorpdfstring{$\epsilon = 0$}{epsilon=0}} 

The last result we prove focuses on the special case with $\epsilon = 0$ and illustrates the relationship between our problem and the classical formulation with cost function~\eqref{eq:classical-cost}.
Without loss of generality, we assume $\alpha = 1$ in~\eqref{eq:classical-cost} and by a conditioning argument, we obtain the Bayesian formulation for the  value function for the classical quickest detection problem~\cite[Ch. 22]{peskir2006optimal}:
\begin{align*}
    V_{\rm{c}}(\pi ) = \inf_{\tau \in \cT (\bar{\cF}_t) } \E_\pi \left[  1 - \bar{\Pi}_\tau  + \beta \int_0^\tau \bar{\Pi}_t\ud t   \right],
\end{align*}
It is optimal  to stop whenever $\bar{\Pi}_t$ exceeds a threshold $a^*_{\rm{c}} \in (0, 1)$ such that $\psi(a^*_{\rm{c}}) = -1$~\cite[Sec.\ 22, Thm.\ 22.1]{peskir2006optimal}.
By~\eqref{eq:hitting-Psi},  for $\pi \leq a^*_{\rm{c}}$, we have
\begin{equation}\label{eq:value-c}
    V_{\rm{c}}(\pi ) =  (1 - a^*_{\rm{c}}) + \left\{ \Psi(\pi) - \Psi( a^*_{\rm{c}} ) \right\}.
\end{equation}
For comparison, let $V_0(\pi )$ denote the value function for the recursive quickest detection problem \eqref{eq:def.v} with $\epsilon = 0$, and let $a^*_0$ denote the corresponding optimal stopping boundary. By~\eqref{eq:hatV-alter}, for $\pi \leq  a^*_0$,
\begin{equation}\label{eq:value-0}
V_0(\pi ) =   \frac{1}{ a^*_0} + \left\{ \Psi(\pi) -  \frac{1 }{ a^*_0 } \Psi(a^*_0) \right\}.
\end{equation}
While the functions~\eqref{eq:value-c} and~\eqref{eq:value-0} are not directly comparable, after a suitable scaling of the model parameter $\beta$, the two stopping boundaries $a^*_{\rm{c}}, a^*_0$ coincide, and the two value functions can be obtained from each other by a simple linear transformation.
This connection is formally established in the next proposition, where we write $V_{\rm{c}}(\pi; \beta), V_{0}(\pi; \beta), a^*_{\rm{c}}(\beta), a^*_0(\beta), \Psi(\pi; \beta)$ to emphasize the dependence on $\beta$.

\begin{proposition}\label{prop:recur2classical}
Given $\beta>0$, it holds $V_0(\pi; \beta) = 1+V_0(0;\beta)V_{\rm{c}}(\pi; \tilde{\beta})$
 and $a^*_{\rm{c}}(\tilde{\beta}) = a^*_0(\beta)$
with $\tilde{\beta} = \beta / V_0(0; \beta)$.
\end{proposition}
\begin{proof}
By the definition of $V$ in \eqref{eq:def.v}, recalling $g_0(\pi)=0$ from \eqref{eq:def-g} and using $\tilde{\beta} = \beta / V_0(0; \beta)$,
\begin{align*}
V_0(\pi; \beta ) &=  \inf_{\tau \in \cT (\bar{\cF}_t) } \E_\pi \Big[ 1 +  \beta \int_0^\tau \Pi_t \ud t + (1 - \Pi_\tau) V_0(0; \beta)   \Big]  \\
&= 1 + V_0(0; \beta) \inf_{\tau \in \cT (\bar{\cF}_t) } \E_\pi \Big[\frac{\beta}{V_0(0; \beta)} \int_0^\tau \Pi_t\ud t  + 1 - \Pi_\tau \Big]  \\
&= 1 + V_0(0; \beta)  V_{\rm{c}}(\pi; \tilde{\beta}).
\end{align*}
Equivalence of the optimal stopping rules is now clear because
\[
V_0(\pi; \beta ) = 1+(1-\pi)V_0(0;\beta)\iff V_{\rm{c}}(\pi;\tilde\beta)=1-\pi
\]
shows equivalence of the stopping sets.  (One can also verify $a^*_{\rm{c}}(\tilde{\beta}) = a^*_0(\beta)$ by directly comparing the equations they must satisfy.)
\end{proof}

\section{Discussion and further generalization} \label{sec:disc}

The above analysis of the special case with $\epsilon = 0$ illustrates that the cost function~\eqref{eq:classical-cost} for the classical quickest detection problem and our cost function~\eqref{eq:our-cost}  are quite similar, in the sense that both penalize delay in detection and  stopping before the drift occurs.  Allowing for $\epsilon\in(0,1)$, our cost function takes into account the possibility of false negatives, which is the primary motivation of this work.
{We acknowledge that the formulation studied in this paper is  simplified relative to many real-world applications, which makes it possible to derive an explicit solution. This model also serves as a useful benchmark for possible, more complex extensions, which we now discuss to conclude the paper.}

A natural extension is to combine the two cost functions to allow for both ``inspections'' and ``declaration'' of the drift. This can be useful in some scenarios where system inspections are costly and one may prefer to simply assume the drift has occurred; in the context of the COVID-19 example, this could correspond to initiating the treatment without confirming the disease through testing.
Mathematically, the optimizer can choose a sequence of nondecreasing  random times $(\tau_j)_{j \geq 1}$, and at each $\tau_k$ the optimizer can take an action $d_k \in \{0, 1\}$. If $d_k = 0$,  an inspection is performed on the system and the optimizer observes $Z_k \in \{0, 1\}$ (the distribution of $Z_k$ is the same as in our problem formulation), and if $d_k = 1$, the optimizer declares that the drift has occurred and ends the optimization  (without inspection).
The goal is to find the optimal policy $(\tau_j, d_j)_{j \geq 1}$ that minimizes the cost function
\begin{equation}\label{eq:combined-cost}
\alpha \P_\pi (\tau_N < \theta) +    \E_\pi\Big[ \beta  (\tau_N -\theta)^+\Big]
   + \E_\pi\Big[\nu \sum\nolimits_{j=1}^N  1_{\{ d_j = 0\} } \Big],
\end{equation}
where
\begin{equation}
    N = \min\{k\in\N \colon  (1 - d_k) Z_k = 1 \text{ or } d_k = 1 \},
\end{equation}
$\alpha > 0$ measures the cost of a false declaration of the drift, and $\nu > 0$ measures the cost of each inspection.

If $d_N = 0$,   the optimization ends because the drift is detected after an inspection. In that case $\P (\tau_N < \theta)  = 0$.  If instead $d_N=1$, the optimization ends because a drift has been declared without inspection. In that case $\P (\tau_N < \theta) $ may be strictly positive.
The classical quickest detection problem can be seen as a special case with $\nu\rightarrow \infty$, because each inspection becomes prohibitively expensive. In contrast, our formulation is a special case with  $\alpha \rightarrow \infty$, i.e., when the cost of mistakenly declaring a drift is infinite. A comprehensive investigation into this problem is left for future research.

The cost function~\eqref{eq:combined-cost} also enables us to further generalize the problem by allowing for false positives. That is, we can allow both probabilities
\[
\P_\pi(Z = 1 \mid \theta > \tau)\ \text{ and }\ \P_\pi(Z = 0 \mid \theta \leq \tau)
\]
to be strictly positive, where $Z$ denotes the inspection outcome at  time $\tau$.
The cost function~\eqref{eq:combined-cost} can still be used, and in this case, even if the optimization stops due to a positive inspection outcome, the false declaration probability $\P_\pi (\tau_N < \theta)$ is not zero.
The methodology employed in this paper may  perhaps be reused to address this problem, though the solution is likely to be much more complicated to characterize.
{Other possible extensions include incorporating costly observations, as in Bayraktar {\rm et al.}~\cite{BEG2022}, where the optimizer pays a fee each time the signal process $X$ is observed,
and introducing budget constraints that limit the total number of inspections or the total expenditure that can be incurred.
Such extensions would further enrich the model by capturing realistic resource limitations, which we  leave for future research.
}

\section*{Acknowledgements}
 We would like to thank the anonymous reviewers for valuable feedback, especially for suggesting the alternative problem formulation that we discussed in Section~\ref{sec:disc}.
The work of J.\ Garg and Q.\ Zhou was partially supported by the National Science Foundation through grants DMS-2245591, DMS-2311307. T.\ De Angelis was partially supported by EU -- Next Generation EU -- PRIN2022 (2022BEMMLZ) CUP: D53D23005780006 and by PRIN-PNRR2022 (P20224TM7Z) CUP: D53D23018780001.

\renewcommand{\thesection}{\Alph{section}} 
\setcounter{section}{0} 
\section{Appendix}
\subsection{Proof of Proposition~\ref{prop:dynamics}\label{sec:proof.dynamics}}
\begin{proof}
Fix $k$ and let $A\in\cG^k_{\tau_k\vee t}$. Then, using the definition of $\Pi^k$ and of the measure $\P_\pi$, we obtain that
\begin{align*}
&\E_\pi\big[1_A\Pi^k_{\tau_k\vee t}\big]=\E_\pi\big[1_A1_{\{\theta\le \tau_k\vee t\}}\big]\\
&=\pi\E^0\big[1_A\big]+(1-\pi)\int_0^\infty\lambda \e^{-\lambda s}\E^s\big[1_A1_{\{\theta\le \tau_k\vee t\}}\big]\ud s\\
&=\E_\pi\Big[1_A\Big(\pi\frac{\ud \P^0}{\ud \P_\pi}\Big|_{\cG^k_{\tau_k\vee t}}+(1-\pi)\int_0^{\tau_k\vee t}\lambda \e^{-\lambda s}\frac{\ud \P^s}{\ud \P_\pi}\Big|_{\cG^k_{\tau_k\vee t}}\ud s\Big)\Big].
\end{align*}
Thus, $\P_\pi$-a.s.
\[
\Pi^k_{\tau_k\vee t}=\pi\frac{\ud \P^0}{\ud \P_\pi}\Big|_{\cG^k_{\tau_k\vee t}}+(1-\pi)\int_0^{\tau_k\vee t}\lambda \e^{-\lambda s}\frac{\ud \P^s}{\ud \P_\pi}\Big|_{\cG^k_{\tau_k\vee t}}\ud s.
\]
By analogous arguments, for $t>0$
\begin{align*}
&\E_\pi\big[1_A(1-\Pi^k_{\tau_k\vee t})\big]=\E_\pi\big[1_A1_{\{\theta> \tau_k\vee t\}}\big]\\
&=(1-\pi)\int_0^\infty\lambda \e^{-\lambda s}\E^s\big[1_A1_{\{\theta > \tau_k\vee t\}}\big]\ud s\\
&=\E_\pi\Big[1_A(1-\pi)\int_{\tau_k\vee t}^\infty \lambda \e^{-\lambda s}\frac{\ud \P^s}{\ud \P_\pi}\Big|_{\cG^k_{\tau_k\vee t}}\ud s\Big]\\
&=\E_\pi\Big[1_A(1-\pi)\frac{\ud \P^\infty}{\ud \P_\pi}\Big|_{\cG^k_{\tau_k\vee t}}\int_{\tau_k\vee t}^\infty \lambda \e^{-\lambda s}\ud s\Big]\\
&=\E_\pi\Big[1_A(1-\pi)\frac{\ud \P^\infty}{\ud \P_\pi}\Big|_{\cG^k_{\tau_k\vee t}}\e^{-\lambda (\tau_k\vee t)}\Big],
\end{align*}
where in the third equality we use that $\P^t|_{\cG^k_t}=\P^s|_{\cG^k_t}=\P^\infty|_{\cG^k_t}$ for any $t\le s$. Thus, $\P_\pi$-a.s.
\[
1-\Pi^k_{\tau_k\vee t}=(1-\pi)\frac{\ud \P^\infty}{\ud \P_\pi}\Big|_{\cG^k_{\tau_k\vee t}}\e^{-\lambda (\tau_k\vee t)}.
\]

In order to find the dynamics of $\Pi^k$ we start by determining $\Pi^k_{\tau_k}$, separately on the two events $\{Y_{\tau_k}=0\}$ and $\{Y_{\tau_k}>0\}$. First we notice that for any $A\in\cG^k_{\tau_k}$
\[
\E_\pi\big[1_A1_{\{Y_{\tau_k}>0\}}\Pi^k_{\tau_k}\big]=\E_\pi\big[1_A1_{\{Y_{\tau_k}>0\}}1_{\{\theta\le \tau_k\}}\big]=\E_\pi\big[1_A1_{\{Y_{\tau_k>}0\}}\big],
\]
so that $\Pi^k_{\tau_k}=1$, $\P_\pi$-a.s.\ on the event $\{Y_{\tau_k}>0\}$. Next we want to calculate the expressions $\E_\pi\big[1_A1_{\{Y_{\tau_k}=0\}}\Pi^k_{\tau_k}\big]$. Notice that $\cG^k_{\tau_k}$ is generated by sets of the form $F\cap G$ with $F\in\cG^{k-1}_{\tau_k}$ and $G\in\sigma(Z_k)$. Then, in particular, it is sufficient to use such sets of the form $A=F\cap G$. Given that $\sigma(Z_k)$ is finite  and $\{Y_{\tau_k}=0\}\subset\{Z_k=0\}$  we only need to consider $A=F\cap\{Z_k=0\}$, which yields
\begin{align}
\E_\pi\big[1_A1_{\{Y_{\tau_k}=0\}}\Pi^k_{\tau_k}\big]=\E_\pi\big[1_F 1_{\{Y_{\tau_k}=0\}}\Pi^k_{\tau_k}\big]
\end{align}
for $F\in\cG^{k-1}_{\tau_k}$. Then, since $\cG^{k-1}_{\tau_k} \subset \cG^{k}_{\tau_k}$, the tower property yields
\begin{align}\label{eq:change}
\begin{aligned}
&\E_\pi\big[1_F 1_{\{Y_{\tau_k}=0\}}\Pi^k_{\tau_k}\big]=\E_\pi\big[1_F 1_{\{Y_{\tau_k}=0\}}1_{\{\theta\le \tau_k\}}\big]\\
&=\E_\pi\big[1_F \P_\pi\big(Y_{\tau_k}=0, \theta\le \tau_k\big|\cG^{k-1}_{\tau_k}\big)\big] \\
&=\E_\pi\big[1_F1_{\{Y_{\tau_{k-1}}=0\}}\P_\pi\big(Z_k=0,\theta\le \tau_k\big|\cG^{k-1}_{\tau_k}\big)\big]\\
&=\E_\pi\big[1_F1_{\{Y_{\tau_{k-1}}=0\}}\P_\pi(Z_k=0\big|\cG^{k-1}_{\tau_k})\P_\pi\big(\theta\le \tau_k\big|\cG^{k-1}_{\tau_k},Z_k=0\big)\big].
\end{aligned}
\end{align}

The expression under expectation can be calculated by first noticing that
\begin{equation}\label{eq:Pi.k.tau.k}
\begin{aligned}
   &\P_\pi( \theta \leq \tau_k |\cG^{k-1}_{\tau_k}, Z_k=0) \\ 
   &= \frac{ \P_\pi( \theta \leq \tau_k, Z_k=0 |\cG^{k-1}_{\tau_k}) }{  \P_\pi(Z_k=0|\cG^{k-1}_{\tau_k}) } \\
    &= \frac{ \P_\pi( Z_k=0 |\cG^{k-1}_{\tau_k},\theta \leq \tau_k) \P_\pi(\theta \leq \tau_k|\cG^{k-1}_{\tau_k} ) }{\P_\pi(Z_k=0|\cG^{k-1}_{\tau_k})}=\frac{\epsilon\Pi^{k-1}_{\tau_k}}{\P_\pi(Z_k=0|\cG^{k-1}_{\tau_k})},
\end{aligned}
\end{equation}
where the final expression holds by \eqref{eq:false.negative} and the definition of $\Pi^{k-1}$.
For the denominator of the expression above we use
\begin{align}\label{eq:YkGk}
\begin{aligned}
\P_\pi(Z_{k}=0|\cG^{k-1}_{\tau_k})&=\P_\pi(\theta\le \tau_k, Z_{k}=0|\cG^{k-1}_{\tau_k})+\P_\pi( \theta> \tau_k, Z_{k}=0|\cG^{k-1}_{\tau_k})\\
&= \epsilon\Pi^{k-1}_{\tau_k} +\P_\pi(\theta> \tau_k|\cG^{k-1}_{\tau_k})\\
&=1-(1-\epsilon)\Pi^{k-1}_{\tau_k},
\end{aligned}
\end{align}
where the second equality uses  the numerator in  \eqref{eq:Pi.k.tau.k} and $\{\theta>\tau_k\}\subset\{Z_k=0\}$,
and the final expression holds by definition of $\Pi^{k-1}$.
Combining \eqref{eq:Pi.k.tau.k} and \eqref{eq:YkGk} we obtain that
\begin{align}\label{eq:gmeaning}
\P_\pi( \theta \leq \tau_k |\cG^{k-1}_{\tau_k}, Z_k=0)=g\big(\Pi^{k-1}_{\tau_k}\big).
\end{align}
Substituting into the final line of \eqref{eq:change} we obtain
\begin{equation}
\begin{aligned}
&\E_\pi\big[1_F 1_{\{Y_{\tau_k}=0\}}\Pi^k_{\tau_k}\big] =\E_\pi\big[1_F1_{\{Y_{\tau_{k-1}}=0\}}\P_\pi(Z_k=0\big|\cG^{k-1}_{\tau_k})g\big(\Pi^{k-1}_{\tau_k}\big)\big]\\
&=\E_\pi\big[1_F 1_{\{Y_{\tau_k}=0\}}g\big(\Pi^{k-1}_{\tau_k}\big)\big],
\end{aligned}
\end{equation}
where the final equality holds because $F\cap\{Y_{\tau_k-1}=0\}\in\cG^{k-1}_{\tau_k}$ and $\{Y_{\tau_k-1}=0\}\cap\{Z_k=0\}=\{Y_{\tau_k}=0\}$.
Thus we conclude that $\Pi^k_{\tau_k}=g(\Pi^{k-1}_{\tau_k})$, $\P_\pi$-a.s.\ on the event $\{Y_{\tau_k}=0\}$.

Next, in order to obtain \eqref{eq:Pik} we follow the usual route and introduce the likelihood ratio
\[
\Phi^k_{t\vee\tau_k}=\frac{\Pi^k_{t\vee\tau_k}}{1-\Pi^k_{t\vee\tau_k}}=\e^{\lambda(\tau_k\vee t)}\frac{\ud \P^0}{\ud \P^\infty}\Big|_{\cG^k_{t\vee\tau_k}}\Big(\phi+\int_0^{t\vee\tau_k}\lambda \e^{-\lambda s}\frac{\ud \P^s}{\ud \P^0}\Big|_{\cG^k_{t\vee\tau_k}}\ud s\Big),
\]
where $\phi=\pi(1-\pi)^{-1}$. By construction we have $\Phi^k_{\tau_k}=\infty$ on $\{Y_{\tau_k}>0\}$ and
\begin{align}\label{eq:Phi0}
1_{\{Y_{\tau_k}=0\}}\Phi^k_{\tau_k}=1_{\{Y_{\tau_k}=0\}}\frac{g\big(\Pi^{k-1}_{\tau_k}\big)}{1-g\big(\Pi^{k-1}_{\tau_k}\big)}.
\end{align}
It is not difficult to show by Girsanov's theorem that for any $t\ge 0$ and any $k\in\N\cup\{0\}$
\[
\frac{\ud \P^0}{\ud \P^\infty}\Big|_{\cG^k_t}=\exp\Big(\frac{\mu}{\sigma^2}\big(X_t-X_0\big)-\frac{\mu^2}{2\sigma^2}t\Big)=:M_t.
\]
Similarly, for $s\in[0,t]$,
\[
\frac{\ud \P^0}{\ud \P^s}\Big|_{\cG^k_t}=\exp\Big(\frac{\mu}{\sigma^2}\big(X_s-X_0\big)-\frac{\mu^2}{2\sigma^2}s\Big)=M_s.
\]
Using these expressions in the formula for $\Phi^k$ we arrive at
\begin{align}
\Phi^k_{t\vee\tau_k}=e^{\lambda (t\vee\tau_k)}M_{t\vee\tau_k}\Big(\phi+\int_0^{t\vee\tau_k}\lambda \e^{-\lambda s}\frac{1}{M_s}\ud s\Big).
\end{align}
By It\^o's formula
\[
\ud M_t=\frac{\mu}{\sigma^2}M_t\ud X_t
\]
and therefore, on the event $\{Y_{\tau_k}=0\}$
\[
\Phi^k_{t\vee\tau_k}=\frac{g\big(\Pi^{k-1}_{\tau_k}\big)}{1-g\big(\Pi^{k-1}_{\tau_k}\big)}+\int_{\tau_k}^{t\vee \tau_k}\lambda\big(1+\Phi^k_s\big)\ud s+\int_{\tau_k}^{t\vee \tau_k}\frac{\mu}{\sigma^2}\Phi^k_s\ud X_s
\]
and
\[
\Pi^k_{t\vee\tau_k}=g\big(\Pi^{k-1}_{\tau_k}\big)+\int_{\tau_k}^{t\vee \tau_k}\lambda\big(1-\Pi^k_s\big)\ud s+\int_{\tau_k}^{t\vee \tau_k}\frac{\mu}{\sigma^2}\Pi^k_s\big(1-\Pi^k_s\big)\big(\ud X_s-\mu \Pi^k_s\ud s\big).
\]

It remains to show that
\[
\widetilde W^k_t:=W^k_{\tau_k+t}-W^k_{\tau_k}=\frac{1}{\sigma}\int_{\tau_k}^{\tau_k+t}\big(\ud X_s-\mu \Pi^k_s\ud s\big),
\]
is indeed a Brownian motion with respect to the filtration $(\cG^k_{\tau_k+t})_{t\ge 0}$  under $\P_\pi$.
As usual this will be accomplished via L\`evy's theorem \cite[Thm.\ 3.3.16]{karatzas2012brownian}.

It is clear that $\widetilde W^k_0=0$ and that $\E_\pi[|\widetilde W^k_t|]<\infty$ for all $t\ge 0$.
Moreover,
\[
\langle \widetilde W^k\rangle_t=\sigma^{-2}\langle X_{\tau_k+\cdot}-X_{\tau_k}\rangle_t=t
\]
and it only remains to check the martingale property. Take $s<t$:
\begin{equation}
\begin{aligned}
\E_\pi\big[\widetilde W^k_t-\widetilde W^k_s\big|\cG^k_{\tau_k+s}\big]=&\frac{1}{\sigma}\E_\pi\Big[\int_{\tau_k+s}^{\tau_k+t}\Big(\ud X_r-\mu\P_\pi\big(\theta\le r|\cG^k_{r}\big)\ud r\Big)\Big|\cG^k_{\tau_k+s}\Big]\\
=&\frac{1}{\sigma}\E_\pi\Big[\int_{\tau_k+s}^{\tau_k+t}\Big(\ud X_r-\mu1_{\{\theta\le r\}}\ud r\Big)\Big|\cG^k_{\tau_k+s}\Big]\\
=&\E_\pi\Big[B_{\tau_k+t}-B_{\tau_k+s}\Big|\cG^k_{\tau_k+s}\Big],
\end{aligned}
\end{equation}
where the second equality is by tower property and observing that $\{\tau_k+s\le r\}$ and $\{\tau_k+t>r\}$ are both $\cG^k_r$-measurable.
Increments of the Brownian motion $B$ after time $\tau_k+s$ are independent of the history of the processes $X$ and $Y$ and independent of $\theta$. Then $\E_\pi\big[B_{\tau_k+t}-B_{\tau_k+s}\big|\cG^k_{\tau_k+s}\big]=0$, concluding our proof. 
\end{proof}

\subsection{Proof of \texorpdfstring{$V(\pi)=\widetilde V^k(\pi)$}{v(pi)=vk(pi)}  }\label{app:VV}
\begin{proof}
Fix an admissible sequence $(\tau_n)_{n=0}^\infty\in\cT_*$. For any $k$, letting $\widetilde\Pi^k_t:=\Pi^{k}_{\tau_k+ t}$, $\widetilde W^k_t:=W^k_{\tau_k+t}-W^k_{\tau_k}$ and $\widetilde \cG^k_t:=\cG^k_{\tau_k+t}$, the process $((\widetilde\Pi^k_t,\widetilde \cG^k_{t})_{t\ge 0}, \P_\pi)$ is a strong Markov process. In particular, when $\widetilde\Pi^k_0=\pi$ for some $\pi\in[0,1]$ we have $\widetilde\Pi^k_t=f(t,\pi,\widetilde W^k_\cdot)$ for some measurable non-anticipative mapping $f:\R_+\times[0,1]\times \cC(\R_+)\to [0,1]$, because \eqref{eq:Pik} admits a unique strong solution. It is worth noticing that the function $f$ is the same for all $k$'s because the structure of the SDE in \eqref{eq:Pik} is the same for all $k$'s. By the same observation it also follows that $\bar \Pi_t=f(t,\pi,\bar W_t)$. Therefore, for any $\pi,\pi'\in[0,1]$
\begin{align}\label{eq:laws}
\mathsf{Law}^{\P_\pi}\Big((\widetilde \Pi^k_t)_{t\ge 0}\big|\widetilde\Pi^k_0=\pi'\Big)=\mathsf{Law}^{\bar\P}\Big((\bar \Pi_t)_{t\ge 0}\big|\bar\Pi_0=\pi'\Big).
\end{align}

By strong Markov property of $\widetilde\Pi^k$, there exists a Borel measurable function
\[
\widetilde V^k\colon[0,1]\to [0,\infty)
\]
for which
\begin{align}\label{eq:def.vk}
\begin{aligned}
&\essinf_{\substack{\tau\in\cT(\cG^k_t)
\tau\ge\tau_k}} \E_\pi \Big[\beta\int_{\tau_k}^\tau \Pi^k_t\ud t  + (\cA V) \big( \Pi^k_\tau\big)\Big|\cG^k_{\tau_k} \Big]\\
&=\essinf_{\widetilde\tau\in\cT(\widetilde\cG^k_t)
} \E_\pi \Big[\beta\int_{0}^{\widetilde\tau} \widetilde\Pi^k_t\ud t  + (\cA V) \big( \widetilde\Pi^k_\tau\big)\Big|\widetilde\cG^k_{0} \Big]=\widetilde V^k(\widetilde\Pi^k_{0}).
\end{aligned}
\end{align}
Notice that on the event $\{Y_{\tau_k}=0\}$ it holds $\widetilde V^k(\Pi^k_{\tau_k})=\widetilde V^k\big(g(\Pi^{k-1}_{\tau_k})\big)$.
The function, $\widetilde V^k$, is constructed in \cite[Ch.\ 3, Secs.\ 2 and 3]{shiryaev2007optimal} (see Lemma~3 and Theorem~1 therein) as the largest measurable function $\Gamma:[0,1]\to[0,\infty)$ dominated by $\cA V$ (i.e., $\Gamma(\pi)\le(\cA V)(\pi)$ for $\pi\in[0,1]$) that satisfies the sub-harmonic property
\begin{align}\label{eq:subh1}
\Gamma(\widetilde \Pi^k_{\tau})\le \E_\pi\Big[\Gamma(\widetilde \Pi^k_{\sigma})+\beta\int_{\tau}^{\sigma}\widetilde\Pi^k_s\ud s\Big|\widetilde\cG^k_\tau\Big]
\end{align}
for any $\tau,\sigma\in\cT(\widetilde\cG^k_t)$ such that $0\le \tau\le \sigma$. Analogously, the value function $V$ of \eqref{eq:def.v} is the largest measurable function dominated by $\cA V$ that satisfies the sub-harmonic property
\begin{align}\label{eq:subh2}
V(\bar \Pi_{\tau})\le \bar \E_\pi\Big[V(\bar \Pi_{\sigma})+\beta\int_{\tau}^{\sigma}\bar\Pi_s\ud s\Big|\bar\cF_\tau\Big].
\end{align}

We can restrict the class of stopping times to the class of entry times to Borel sets \cite[Ch.\ 3, Sec.\ 3, Thm.\ 3]{shiryaev2007optimal}. Then, the law of the underlying process completely determines $\widetilde V^k$ and $V$ via conditions \eqref{eq:subh1} and \eqref{eq:subh2}. Since \eqref{eq:laws} holds, maximality of the sub-harmonic function dominated by $\cA V$ yields
\begin{align}\label{eq:VkV}
\widetilde V^k(\pi)=V(\pi),\quad \pi\in[0,1],
\end{align}
which completes the proof. 
\end{proof}

\subsection{Proof of Lemma~\ref{lm:psi}\label{sec:proof.lm.psi}}
\begin{proof}
Using $h'(\pi) = \rho \pi^{-2}(1 - \pi)^{-1}$ and integration by parts, we find that
\begin{align}
\rho \int_0^\pi \frac{ \e^{h(x)}}{x(1-x)^2} \ud x
=\;&  \frac{\pi \e^{h(\pi)}}{1 - \pi} - \int_0^\pi \frac{\e^{h(x)}}{(1-x)^2} \ud x,   \label{eq:psi1.1} \\
\int_0^\pi \frac{\e^{h(x)}}{(1-x)^2} \ud x =\;&   \frac{\pi^2e^{h(\pi)}}{\rho(1 - \pi)} - \int_0^\pi \frac{\e^{h(x)}}{\rho }  \left[ \frac{1}{(1-x)^2} - 1 \right] \ud x.  \label{eq:psi1.2}
\end{align}
Since $\pi \in (0, 1)$, all integrals in the above equalities are positive. Hence, using $\rho = \lambda / \gamma$, we get
\begin{equation}\label{eq:psi.bound1}
  -\frac{\pi}{1 - \pi}  < \frac{\lambda}{\beta} \psi(\pi) = -   \rho \e^{- h (\pi) } \int_0^\pi \frac{\e^{h(x)}}{x(1-x)^2} \ud x < -\frac{\pi}{1 - \pi}  + \frac{\pi^2}{\rho (1 - \pi)},
\end{equation}
which are slightly looser than the bounds stated in the lemma. To sharpen the bounds, observe that for any  $\pi \in (0, 1)$,
\begin{align}
    & \int_0^\pi \frac{\e^{h(x)}}{(1-x)^2}  < \int_0^\pi \frac{\e^{h(x)}}{x (1-x)^2} \ud x, \label{eq:psi.b1} \\
    & \int_0^\pi \frac{\e^{h(x)}}{\rho } \left[ \frac{1}{(1-x)^2} - 1 \right]  \ud x
    < \rho^{-1} \int_0^\pi \frac{\e^{h(x)}}{ (1-x)^2}  \ud x.  \label{eq:psi.b2}
\end{align}
Plugging~\eqref{eq:psi.b1} into~\eqref{eq:psi1.1} and~\eqref{eq:psi.b2} into~\eqref{eq:psi1.2}, we obtain
\begin{align}
    \rho \int_0^\pi \frac{ \e^{h(x)}}{x(1-x)^2} >\;&  \frac{\rho}{1 + \rho} \frac{\pi \e^{h(\pi)}}{1 - \pi}, \label{eq:psi.b3} \\
    \int_0^\pi \frac{\e^{h(x)}}{(1-x)^2} \ud x >\;& \frac{1}{ 1 + \rho }  \frac{\pi^2   \e^{h(\pi)}}{1-\pi}. \label{eq:psi.b4}
\end{align}
Plugging~\eqref{eq:psi.b4} into~\eqref{eq:psi1.1}, we get
\begin{equation}\label{eq:psi.b5}
\rho \int_0^\pi \frac{ \e^{h(x)}}{x(1-x)^2} \ud x = \frac{\pi \e^{h(\pi)}}{1 - \pi} - \int_0^\pi \frac{\e^{h(x)}}{(1-x)^2} \ud x
< \frac{\pi \e^{h(\pi)}}{1 - \pi} \left( 1 - \frac{\pi}{1 + \rho} \right),
\end{equation}
The upper bound stated in the lemma follows from~\eqref{eq:psi.b3} and the lower bound follows from~\eqref{eq:psi.b5}.

Letting $\pi \rightarrow 1$, we find that
\[\lim_{\pi \uparrow 1}  \frac{ \lambda  }{ \beta}(1 - \pi) \psi(\pi)  = - \frac{ \rho }{ 1 + \rho} = -\frac{\lambda}{\lambda + \gamma},\]
from which the second claim follows.
\end{proof}

\subsection{Proof of Lemma~\ref{lm:u2}\label{sec:proof.lm.u2}}
\begin{proof}
Plugging $\Psi(\pi) = \int_0^\pi \psi(x) dx$ into the ODE in \eqref{eq:c1} and solving for $\psi'$ we obtain
\begin{align} \label{eq:def.psi1}
\psi'(\pi) =\;& - \frac{\beta \pi + \lambda(1 - \pi) \psi (\pi)}{\gamma  \pi^2 (1 - \pi)^2 }.
\end{align}
Plugging the bounds in~\eqref{eq:psi.bound1} into~\eqref{eq:def.psi1}, we get the desired bounds for $\psi'$.
\end{proof}

\subsection{Proof of Lemma~\ref{lm:key}\label{sec:proof.lm.key}}

\begin{proof}
Consider the first inequality.
Fix an arbitrary $\pi \in (0, 1)$ and set
\begin{align*}
f(\epsilon) = (1 - \epsilon) \beta \pi + \lambda (1 - \pi) [ \psi(\pi) - \psi(g_\epsilon(\pi))].
\end{align*}
Since $g_\epsilon(\pi) = \pi$ when $\epsilon = 1$, we have $f(1) = 0$.
Then it suffices to prove that $f'(\epsilon) < 0$ for $\epsilon \in (0, 1)$.
A direct calculation yields
\begin{align*}
f'(\epsilon) =\;&  - \beta \pi - \lambda (1 - \pi) \psi'( g_\epsilon(\pi) ) \frac{\partial g_\epsilon (\pi)}{\partial \epsilon} \\
=\;&  - \beta \pi - \lambda \psi'( g_\epsilon(\pi) )  \frac{  \pi(1 - \pi)^2 }{ [ 1 - (1 - \epsilon) \pi]^2 } \\
< \;&  - \beta \pi  +   \frac{ \beta \pi(1 - \pi)^2 }{ [ 1 - (1 - \epsilon) \pi]^2 [ 1 - g_\epsilon(\pi) ]^2} = 0,
\end{align*}
where the strict inequality follows from Lemma~\ref{lm:u2} and the final equality follows from the definition of $g_\epsilon$, which yields the equality
\[1 - g_\epsilon(\pi) = \frac{ 1 - \pi } {1 - (1 - \epsilon)\pi}.\]

To prove the second asserted inequality, we use~\eqref{eq:def.psi1} to get
\begin{align} \label{eq:key1}
    \psi(\pi) = -  \frac{ \gamma \pi^2 (1 - \pi)^2 \psi'(\pi) + \beta \pi }{\lambda (1 - \pi)},
\end{align}
\begin{align}\label{eq:key2}
    \psi(g(\pi)) = -  \frac{ \gamma g(\pi)^2 (1 - g(\pi))^2 \psi'(g(\pi)) + \beta g(\pi) }{\lambda (1 - g(\pi))} 
    =     - \frac{ \frac{ \gamma \epsilon^2  \pi^2 (1 - \pi)^2 \psi'(g(\pi)) }{[1 - (1 - \epsilon)\pi ]^3 } + \beta \epsilon \pi  }{\lambda (1-\pi) }.
\end{align}
The proof is completed by plugging~\eqref{eq:key1} and~\eqref{eq:key2} into  $f(\epsilon) > 0$ (i.e., the first part of the lemma).
\end{proof}

\bibliographystyle{abbrv}
\bibliography{ref}

\begin{thebibliography}{10}

\bibitem{antelman1965surveillance}
G.~R. Antelman and I.~R. Savage.
\newblock Surveillance problems: {W}iener processes.
\newblock {\em Naval Research Logistics Quarterly}, 12(1):35--55, 1965.

\bibitem{B75}
D.~W. Balmer.
\newblock On a quickest detection problem with costly information.
\newblock {\em J.\ Appl.\ Probab.}, 12:87--97, 1975.

\bibitem{bartlett2021false}
E.~C. Bartlett, M.~Silva, M.~E. Callister, and A.~Devaraj.
\newblock False-negative results in lung cancer screening—evidence and
  controversies.
\newblock {\em Journal of Thoracic Oncology}, 16(6):912--921, 2021.

\bibitem{BEG2022}
E.~Bayraktar, E.~Ekstr\"om, and J.~Guo.
\newblock Disorder detection with costly observations.
\newblock {\em J.\ Appl.\ Probab.}, 59:338--349, 2022.

\bibitem{bayraktar2015quickest}
E.~Bayraktar and R.~Kravitz.
\newblock Quickest detection with discretely controlled observations.
\newblock {\em Sequential Anal.}, 34(1):77--133, 2015.

\bibitem{borodin2015handbook}
A.~N. Borodin and P.~Salminen.
\newblock {\em Handbook of {B}rownian motion-facts and formulae}.
\newblock Birkh\"auser, second edition, 2015.

\bibitem{carmona2008swing}
R.~Carmona and N.~Touzi.
\newblock Optimal multiple stopping and valuation of swing options.
\newblock {\em Math.\ Finance}, 18(2):239--268, 2008.

\bibitem{colaneri2021class}
K.~Colaneri and T.~De~Angelis.
\newblock A class of recursive optimal stopping problems with applications to
  stock trading.
\newblock {\em Math.\ Oper.\ Res.}, 47:1833--1861, 2022.

\bibitem{DS15}
R.~C. Dalang and A.~N. Shiryaev.
\newblock A quickest detection problem with an observation cost.
\newblock {\em Ann.\ Appl.\ Probab.}, 25:1475--1512, 2015.

\bibitem{elkaroui1981}
N.~El~Karoui.
\newblock Les aspects probabilistes du contr{\^o}le stochastique.
\newblock In {\em {\'E}cole d’{\'e}t{\'e} de Probabilit{\'e}s de Saint-Flour
  IX-1979}, pages 73--238. Springer, 1981.

\bibitem{ernst2022quickest}
P.~A. Ernst and G.~Peskir.
\newblock Quickest real-time detection of a {B}rownian coordinate drift.
\newblock {\em Ann.\ Appl.\ Probab.}, 32(4):2652--2670, 2022.

\bibitem{GP2006}
P.~V. Gapeev and G.~Peskir.
\newblock The {W}iener disorder problem with finite horizon.
\newblock {\em Stochastic Process.\ Appl.}, 116:1770--1791, 2006.

\bibitem{gapeev2013bayesian}
P.~V. Gapeev and A.~N. Shiryaev.
\newblock Bayesian quickest detection problems for some diffusion processes.
\newblock {\em Adv.\ Appl.\ Probab.}, 45(1):164--185, 2013.

\bibitem{PG2022}
K.~Glover and G.~Peskir.
\newblock Quickest detection problems for {O}rnstein-{U}hlenbeck processes.
\newblock {\em Math.\ Oper.\ Res.}, 49:1045--1064, 2024.

\bibitem{johnson2025bayesian}
P.~Johnson and J.~L. Pedersen.
\newblock Bayesian changepoint detection for epidemic models.
\newblock {\em Scientific Reports}, 15(1):20545, 2025.

\bibitem{johnson2017quickest}
P.~Johnson and G.~Peskir.
\newblock Quickest detection problems for {B}essel processes.
\newblock {\em Ann.\ Appl.\ Probab.}, 27(2):1003--1056, 2017.

\bibitem{kanji2021false}
J.~N. Kanji, N.~Zelyas, C.~MacDonald, K.~Pabbaraju, M.~N. Khan, A.~Prasad,
  J.~Hu, M.~Diggle, B.~M. Berenger, and G.~Tipples.
\newblock False negative rate of {COVID-19 PCR} testing: a discordant testing
  analysis.
\newblock {\em Virology journal}, 18(1):1--6, 2021.

\bibitem{karatzas2012brownian}
I.~Karatzas and S.~Shreve.
\newblock {\em Brownian motion and stochastic calculus}, volume 113.
\newblock Springer Science \& Business Media, 2012.

\bibitem{kobylanski2011optimal}
M.~Kobylanski, M.-C. Quenez, and E.~Rouy-Mironescu.
\newblock Optimal multiple stopping time problem.
\newblock {\em Ann.\ Appl. Probab.}, 21(4):1365--1399, 2011.

\bibitem{peskir2006optimal}
G.~Peskir and A.~Shiryaev.
\newblock {\em Optimal stopping and free-boundary problems}.
\newblock Springer, 2006.

\bibitem{poor2008quickest}
H.~V. Poor and O.~Hadjiliadis.
\newblock {\em Quickest detection}.
\newblock Cambridge University Press, 2008.

\bibitem{shiryaev1961a}
A.~N. Shiryaev.
\newblock The problem of the most rapid detection of a disturbance of a
  stationary regime.
\newblock {\em Soviet Math.\ Dokl.}, 2:795--799, 1961.

\bibitem{shiryaev1963}
A.~N. Shiryaev.
\newblock On optimal methods in quickest detection problems.
\newblock {\em Theory Probab.\ Appl.}, 8:22--46, 1963.

\bibitem{shiryaev2007optimal}
A.~N. Shiryaev.
\newblock {\em Optimal stopping rules}, volume~8.
\newblock Springer Science \& Business Media, 2007.

\bibitem{shiryaev2010quickest}
A.~N. Shiryaev.
\newblock Quickest detection problems: {F}ifty years later.
\newblock {\em Sequential {A}nal.}, 29(4):345--385, 2010.

\bibitem{tartakovsky2014sequential}
A.~Tartakovsky, I.~Nikiforov, and M.~Basseville.
\newblock {\em Sequential analysis: Hypothesis testing and changepoint
  detection}.
\newblock CRC press, 2014.

\bibitem{verbeek2018acceptable}
J.~F. Verbeek, M.~J. Roobol, E.~R.~S. Group, et~al.
\newblock What is an acceptable false negative rate in the detection of
  prostate cancer?
\newblock {\em Translational andrology and urology}, 7(1):54, 2018.

\end{thebibliography}

\end{document}